\documentclass[12pt]{article}
\usepackage{amssymb}
\usepackage{amsmath,bm,natbib,graphics,mathrsfs,graphicx,epsfig,subfigure,placeins,indentfirst,setspace,enumerate,enumitem,caption,varioref,booktabs,tabularx,color,epstopdf,multirow}
\setcitestyle{authoryear,open={(},close={)}}
\usepackage{amsthm}
\makeatletter \oddsidemargin  -.1in \evensidemargin -.1in
\begin{document}
	\newcommand{\bea}{\begin{eqnarray}}
		\newcommand{\eea}{\end{eqnarray}}
	\newcommand{\nn}{\nonumber}
	\newcommand{\bee}{\begin{eqnarray*}}
		\newcommand{\eee}{\end{eqnarray*}}
	\newcommand{\lb}{\label}
	\newcommand{\nii}{\noindent}
	\newcommand{\ii}{\indent}
	\newtheorem{theorem}{Theorem}[section]
	\newtheorem{example}{Example}[section]
	\newtheorem{corollary}{Corollary}[section]
	\newtheorem{definition}{Definition}[section]
	\newtheorem{lemma}{Lemma}[section]
	\newtheorem{remark}{Remark}[section]
	\newtheorem{proposition}{Proposition}[section]
	\numberwithin{equation}{section}
	\renewcommand{\qedsymbol}{\rule{0.7em}{0.7em}}
	\renewcommand{\theequation}{\thesection.\arabic{equation}}
	\renewcommand\bibfont{\fontsize{10}{12}\selectfont}
	\setlength{\bibsep}{0.0pt}
		\title{\bf Some Generalized Information and Divergence Generating Functions: Properties, Estimation, Validation and Applications**}
	
		\author{Shital {\bf Saha}$^1$\thanks {Email address: shitalmath@gmail.com,~520MA2012@nitrkl.ac.in}, ~Suchandan {\bf Kayal}$^2$\thanks{Email address: kayals@nitrkl.ac.in,~suchandan.kayal@gmail.com}, ~and N. {\bf Balakrishnan}$^3$\thanks {Corresponding author :   bala@mcmaster.ca
		\newline**It has been accepted on \textbf{Probability in the Engineering and Informational Sciences}.}}

\date{}
\maketitle \noindent {\it $^{1,2}$Department of Mathematics, National Institute of
	Technology Rourkela, Rourkela-769008, Odisha, India} \\
{\it $^{3}$Department of Mathematics and Statistics,
	McMaster University, Hamilton, Ontario L8S 4K1,
	Canada}
\date{}
\maketitle
		\begin{center}
\textbf{Abstract}
		\end{center} 
We propose R\'enyi information generating function and discuss its  properties. { A connection between the R\'enyi information generating function and the diversity index is proposed for discrete type random variables}. The relation between the R\'enyi information generating function and Shannon entropy of order $q>0$ {is established and several bounds are obtained}. The {R\'enyi information generating function} of escort distribution is derived. Furthermore, we introduce R\'enyi divergence information generating function and discuss its effect under monotone transformations. 
We present  non-parametric and parametric estimators of the {R\'enyi information generating function}. A simulation study is carried out and a real data relating to the failure times of electronic components is analyzed. A comparison study between the non-parametric and parametric estimators is made in terms of  the standard deviation, absolute bias, and mean square error. We have observed superior performance for the newly proposed estimators. Some applications of the proposed {R\'enyi information generating function} and {R\'enyi divergence information generating function} are provided. For three coherent systems, we calculate the values of the {R\'enyi information generating function} and other well-established uncertainty measures and similar behaviour of the {R\'enyi information generating function} is observed. Further, a study regarding the usefulness of the {R\'enyi divergence information generating function} and {R\'enyi information generating function} as model selection criteria is conducted. Finally, three chaotic maps are considered and then used to establish a validation of the proposed information generating function. 
		 \\
\textbf{MSCs:} 94A17; 60E15; 62B10.
			
\section{Introduction}
It is well-known that entropy and divergence measures play a pivotal role in different fields of science and technology. For example, in coding theory \cite{farhadi2008robust} used the concept of entropy for robust coding in a class of sources. In statistical mechanics, \cite{kirchanov2008using} adopted generalized entropy to describe quantum dissipative systems. In economics, \cite{rohde2016j} made use of $J$-divergence measure to study economic inequality.  An important generalization of the Shannon entropy is the R\'enyi entropy, which also unifies other entropies like the min-entropy or collision entropy. Consider two absolutely continuous non-negative random variables $X$ and $Y$ with respective  probability density functions (PDFs) $f(\cdot)$ and $g(\cdot)$. Henceforth, the random variables are considered to be non-negative and absolutely continuous, unless  otherwise stated. 
The R\'enyi entropy of $X$ and R\'enyi divergence between $X$ and $Y$ are, respectively, given by (see \cite{renyi1961measures})
\begin{align}\label{eq1.2}
H_\alpha(X)=\delta(\alpha)\log\int_{0}^{\infty}f^\alpha(x)dx~ \text
{and}~ RD^\alpha(X,Y)=\delta^*(\alpha)\log\int_{0}^{\infty}f^\alpha(x)g^{1-\alpha}(x)dx,
\end{align}
where $\delta(\alpha)=\frac{1}{1-\alpha}$, $\delta^*(\alpha)=\frac{1}{\alpha-1}$,  $0<\alpha<\infty,~\alpha\ne1$.
Throughout the paper, `$\log$' is used to denote the natural logarithm. It can be easily established that when $\alpha\rightarrow1$, the R\'enyi entropy and R\'enyi divergence reduce to the Shannon entropy (see \cite{shannon1948mathematical}) and Kullback-Leibler (KL)-divergence (see \cite{kullback1951information}), respectively, given by
\begin{align}\label{eq1.2*}
H(X)=-\int_{0}^{\infty}f(x)\log f(x)dx~ \text
{and}~ KL(X,Y)=\int_{0}^{\infty}f(x)\log\frac{f(x)}{g(x)}dx.
\end{align}

In distribution theory, the properties like mean, variance, skewness, and kurtosis are extracted using successive moments of a probability distribution, which are obtained by taking successive derivatives of the moment generating function at origin. Likewise, the information generating functions (IGFs) for probability distributions are constructed in order to compute many information quantities like the KL-divergence and Shannon entropy.  In Physics and Chemistry, the non-extensive thermodynamics and chaos theory depend on the IGF, also referred to as the entropic moment. In $1966$,  \cite{golomb1966} introduced the IGF and showed that its first-order derivative at $1$ yields negative Shannon entropy. For a  random variable $X$ with PDF  $f(\cdot)$, the Golomb's IGF, for $\gamma>0$, is defined as
\begin{align}\label{eq1.3}
G_\gamma(X)=\int_{0}^{\infty}f^\gamma(x)dx.
\end{align}
It is clear that $G_\gamma(X)\big|_{\gamma=1}=1$ and  $\frac{d}{d\gamma}G_\gamma(X)\big|_{\gamma=1}=-H(X)$. Again, for $\gamma=2$, the IGF in (\ref{eq1.3}) reduces to the Onicescu's informational energy (IE) (see \cite{onicescu1966theorie}), given by 
\begin{eqnarray}\label{eq1.3*}
IE(X)=\int_{0}^{\infty}f^2(x)dx.
\end{eqnarray}
 The IE has many applications in different fields; for example, IE is used as a correlation measure in systems of atoms and molecules (see \cite{flores2016informational}), and highly correlated Hylleraas wave functions in the analysis of the ground state helium (see \cite{ou2019benchmark}). Later, motivated by the Golomb's IGF, \cite{guiasu1985relative} proposed  relative IGF. For random variables $X$ and $Y$, the relative IGF, for $\theta>0$, is given by
\begin{align}
RI_\theta(X,Y)=\int_{0}^{\infty}f^\theta(x)g^{1-\theta}(x)dx.
\end{align}
Apparently, $RI_\theta(X,Y)|_{\theta=1}=1$ and $\frac{d}{d\theta}RI_\theta(X,Y)|_{\theta=1}=KL(X,Y)$. Recently, the IGFs have been studied in great detail due to their capability of generating various useful uncertainty and divergence measures. 

\cite{kharazmi2021jensen} introduced Jensen IGF and IGF for residual lifetime and discussed their important properties.  \cite{kharazmi2022generating} introduced generating function for the  generalised Fisher information and establish various results using it. \cite{kharazmi2023cumulative} proposed cumulative residual IGF and relative cumulative residual IGF. In addition to these, one may also refer to \cite{kharazmi2021cumulative}, \cite{zamani2022information}, \cite{kharazmi2023jensen}, \cite{kharazmi2023optimal}, \cite{smitha2023dynamic}, \cite{kharazmi2023jensen}, \cite{kattumannil2023entropy} and \cite{capaldo2024} for more work on generating functions. Recently, \cite{saha2024general} proposed general weighted IGF and general weighted relative IGF and developed some associated results. 

Motivated by the usefulness of the previously introduced IGFs as described above, we develop here  some IGFs, and explore their properties. We mention that the IGFs with utilities were introduced earlier by \cite{jain2009some} only for discrete cases. In this paper, we mainly focus on the generalized versions of the IGFs in the continuous framework. The key contributions made here are described below:
\begin{itemize}
\item In Section \ref{sec2}, we propose R\'enyi information generating function (RIGF) for both discrete and continuous random variables and discuss various properties. For discrete distributions, a relation between the RIGF and the Hill number (a diversity index) is obtained. The RIGF is expressed in terms of the Shannon entropy of order $q>0$. We also obtain bound for RIGF. The RIGF is then evaluated for escort distributions;
\item In Section \ref{sec3}, we introduced the R\'enyi divergence information generating function (RDIGF). The relation between the RDIGF of generalized escort distributions, the RDIGF and RIGF of baseline distributions is then established. Further, the RDIGF is examined under strictly monotone transformations;
\item In Section \ref{sec5}, we propose  non-parametric and  parametric estimators of  the proposed RIGF and IGF due to \cite{golomb1966}. A Monte Carlo simulation study is carried out for both these estimators. Further, the non-parametric and parametric estimators are compared on the basis of standard deviation (SD), absolute bias (AB), and mean square error (MSE) for the case of Weibull distribution. A real data set is considered and analysed finally in Section \ref{sec6}

\item Section \ref{sec7} discusses some applications of the proposed generating functions. The RIGF is studied for coherent systems. Several properties including bounds are obtained. In particular, three coherent systems are considered, and then the numerical values of RIGF, IGF, R\'enyi entropy, and varentropy are computed for them. It is observed that the proposed measure can be considered as an alternative uncertainty measure since it has similar behaviour as other well-established information measures. Further, we have established that the RDIGF and RIGF can be considered as effective tools for model selection. Furthermore, three chaotic maps, namely, logistic map, Chebyshev map and H\'ennon map, have been considered for the validation of the proposed information generating function. Finally, Section \ref{sec8} presents some concluding remarks. 

 \end{itemize} 

Throughout the paper, all the integrations and differentiations involved are assumed to exist.

\section{R\'enyi information generating functions}\label{sec2}
We propose RIGFs for  discrete and continuous random variables, and discuss some of their properties. First, we present RIGF for a discrete random variable. Hereafter, $\mathbb{N}$ is used  to denote the set of natural numbers.
\begin{definition}\label{def2.1}
Suppose $X$ is a discrete random variable taking values $x_i,$ for $i=1,\dots,n\in\mathbb{N}$ with PMF $P(X=x_i)=p_i>0$, $\sum_{i=1}^{n}p_{i}=1$. Then, the RIGF of $X$ is defined as
\begin{align}\label{eq2.1}
R^\alpha_\beta(X)=\delta(\alpha)\left(\sum_{i=1}^{n}p_i^\alpha\right)^{\beta-1},~~0<\alpha<\infty,~ \alpha\neq1,~ \beta>0,
\end{align} 
where $\delta(\alpha)=\frac{1}{1-\alpha}$.
\end{definition}
Clearly, $R_\beta^\alpha(X)|_{\beta=1}=\delta(\alpha)$  and $R_\beta^\alpha(X)|_{\beta=2,\alpha=2}=-\sum_{i=1}^{n}p_{i}^2=-S$, where $S$ is known as the Simpson's index (see \cite{gress2024mathematical}). We recall that the Simpson's index is useful in ecology to quantify the biodiversity of a habitant. In addition, the proposed RIGF given in (\ref{eq2.1}) can be connected with the Hill number (see \cite{hill1973diversity}), which is also an important diversity index employed by many researchers in ecology (see \cite{chao2010phylogenetic}, \cite{chao2014rarefaction}, \cite{ohlmann2019diversity}). Consider an ecological community containing up to $n$ distinct species, say $x_{i}$ according to a certain process $X$, in which the relative abundance of species $i$ is $p_i$, for $i = 1,\cdots,n$ with $\sum_{i=1}^{n}p_i=1$. Then, the Hill number of order $\alpha$ is defined as
\begin{eqnarray}\label{eq2.1*}
	D_{\alpha}(X)=\left(\sum_{i=1}^{n}p_{i}^{\alpha}\right)^{\frac{1}{1-\alpha}},~~\alpha>0,~\alpha\ne1.
\end{eqnarray}
Thus, from (\ref{eq2.1}) and (\ref{eq2.1*}), we obtain a relation between the RIGF and Hill number of order $\alpha$ as:
\begin{eqnarray}
R^\alpha_\beta(X)=\delta(\alpha)\left(D_{\alpha}\right)^{\frac{\beta-1}{\delta(\alpha)}},~~\alpha>0,~\alpha\ne1,~\beta>0.
\end{eqnarray}

 Further, the $p$th order derivative of $R_\beta^\alpha(X)$ with respect to $\beta$ is obtained as
\begin{align}\label{eq2.2}
\frac{\partial^{p}R_\beta^\alpha(X)}{\partial \beta^p}=\delta(\alpha)\left(\sum_{i=1}^{n}p_i^\alpha\right)^{\beta-1}\left(\log\sum_{i=1}^{n}p_i^\alpha\right)^p,
\end{align}
provided that the sum in (\ref{eq2.2}) is convergent. In particular, 
\begin{align*}
\frac{\partial R_\beta^\alpha(X)}{\partial \beta}\Big|_{\beta=1}=\delta(\alpha)\log\sum_{i=1}^{n}p_i^\alpha
\end{align*}
is the R\'enyi entropy of the discrete  random variable $X$ in Definition \ref{def2.1}. Next, we obtain closed-form expressions of the R\'enyi entropy for some discrete distributions (see Table \ref{tb1}) using the proposed RIGF in (\ref{eq2.1}). We mention here that the RIGF is a simple tool to obtain the  R\'enyi entropy of probability distributions.

\begin{table}[h!]
\caption {The RIGF and R\'enyi entropy of some discrete distributions.}
	\centering 
	\scalebox{1.1}{\begin{tabular}{c c c c c c c c } 
			\hline\hline\vspace{.1cm} 
			PMF & RIGF & R\'enyi entropy \\
			\hline
				$p_i=\frac{1}{n}, ~i=1,2,\dots,n\in\mathbb{N}$&$\delta(\alpha)n^{(1-\alpha)(\beta-1)}$ &$\log n$	\\	[1EX]
				$p_i=ba^i,~ a+b=1,~i=0,1,\cdots$& $\delta(\alpha)\left(\frac{b^\alpha}{1-a^\alpha}\right)^{\beta-1}$ & $\delta(\alpha)\log\frac{b^\alpha}{1-a^\alpha}$\\[1EX]
$p_i=\frac{i^{-a}}{\phi(a)},~a>1;~\phi(a)=\sum_{i=1}^{\infty}i^{-a},~i=1,2,\dots$& $\delta(\alpha)\left(\frac{\phi(\alpha a)}{\phi^\alpha(a)}\right)^{\beta-1}$&$\delta(\alpha)\log\frac{\phi(\alpha a)}{\phi^\alpha(a)}$\\[1EX]
	\hline	 		
	\end{tabular}} 
	\label{tb1} 
\end{table}

Next, we introduce the RIGF for a continuous random variable.
\begin{definition}\label{def2.2}
Let $X$ be a continuous random variable with PDF $f(\cdot)$. Then, for $0<\alpha<\infty,~ \alpha\neq1,~ \beta>0$, the RIGF of $X$ is
\begin{align}\label{eq2.4}
R^\alpha_\beta(X)=\delta(\alpha)\left(\int_{0}^{\infty}f^\alpha(x)dx\right)^{\beta-1}=\delta(\alpha)\left[E
(f^{\alpha-1}(X))\right]^{\beta-1},
\end{align}
where $\delta(\alpha)=\frac{1}{1-\alpha}$.
\end{definition}
Note that the integral in (\ref{eq2.4}) is convergent. The  derivative of $(\ref{eq2.4})$ with respect to $\beta$ is 
\begin{align*}
\frac{\partial R^\alpha_\beta(X)}{\partial\beta}=\delta(\alpha)\left(\int_{0}^{\infty}f^\alpha(x)dx\right)^{\beta-1}\log \int_{0}^{\infty}f^\alpha(x)dx,
\end{align*}
and consequently the $p$th order derivative of RIGF, also known as the $p$th entropic moment, is obtained as 
\begin{align*}
\frac{\partial^p R^\alpha_\beta(X)}{\partial\beta^p}=\delta(\alpha)\left(\int_{0}^{\infty}f^\alpha(x)dx\right)^{\beta-1}\left(\log \int_{0}^{\infty}f^\alpha(x)dx\right)^p.
\end{align*} 
We notice that the RIGF is convex with respect to $\beta$ for $\alpha<1$ and concave for $\alpha>1$. Some important observations of the proposed RIGF are as follows:
\begin{itemize}
\item $R^\alpha_\beta(X)\big|_{\beta=1}=\delta(\alpha)$;~~  $\frac{\partial R^\alpha_\beta(X)}{\partial \beta}\Big|_{\beta=1}=H_{\alpha}(X)$, where $H_{\alpha}(X)$ is as in (\ref{eq1.2});
\item $R^\alpha_\beta(X)\big|_{\beta=2, \alpha=2}=-IE(X)$, where $IE(X)$ is the informational energy, given in (\ref{eq1.3*}).
\end{itemize} 

The expressions of the RIGF and R\'enyi entropy for some continuous distributions are presented in Table \ref{tb2}. Here, $\Gamma(\cdot)$ denotes the complete gamma function. To observe the behaviour of the RIGF of different distributions in Table \ref{tb2} with respect to $\beta$, some graphical plots are presented in Figure \ref{fig1}. From these figures, we notice that they are increasing with respect to $\beta$ for fixed $\alpha$. Also, we observe from the graphs that the RIGF is concave when $\alpha=0.7 (<1)$ and convex when $\alpha=1.5 (>1).$  
\begin{table}[h!]
\caption {The RIGF and R\'enyi entropy for uniform, exponential, and Weibull distributions. For convenience, we denote $\omega_1=\frac{\alpha(c-1)+1}{c}.$}
	\centering 
	\scalebox{1.04}{\begin{tabular}{c c c c c c c c } 
			\hline\hline\vspace{.1cm} 
			PDF  & RIGF & R\'enyi entropy \\
			\hline
				$f(x)=\frac{1}{b-a},~x\in(a,b)$&$\delta(\alpha)(b-a)^{(1-\alpha)(\beta-1)}$ &$\log \left(b-a\right)$	\\	[1EX]
				$f(x)=\lambda e^{-\lambda x}, ~x\ge0$, $\lambda>0$& $\frac{\delta(\alpha)\lambda^{(\alpha-1)(\beta-1)}}{\alpha^{\beta-1}}$ & $\delta(\alpha)\log(\frac{\lambda^{\alpha-1}}{\alpha})$\\[1EX]
$f(x)=cx^{c-1}e^{-x^c},~ x\ge 0,~ c>1$& $\delta(\alpha)(\frac{c^{\alpha-1}}{\alpha^{\omega_1}}\Gamma(\omega_1))^{\beta-1}$&$\delta(\alpha)\log(\frac{c^{\alpha-1}}{\alpha^{\omega_1}}\Gamma(\omega_1))$\\[1EX]
	\hline	 		
	\end{tabular}} 
	\label{tb2} 
\end{table}
%

\begin{figure}[h!]
		\centering
\subfigure[]{\label{c1}\includegraphics[height=2in]{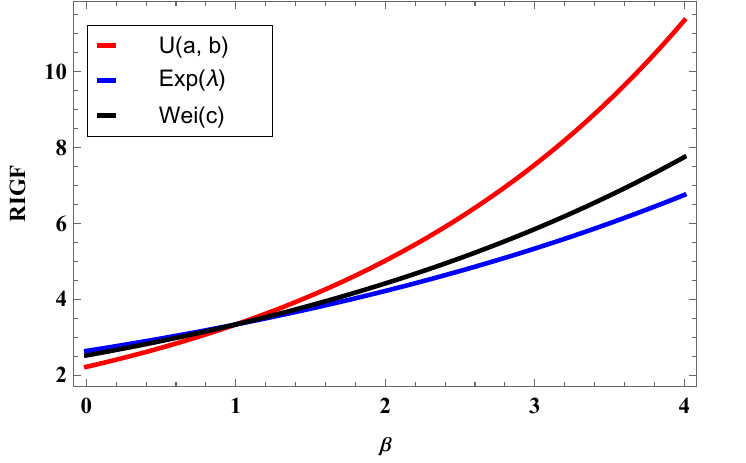}}
\subfigure[]{\label{c1}\includegraphics[height=2in]{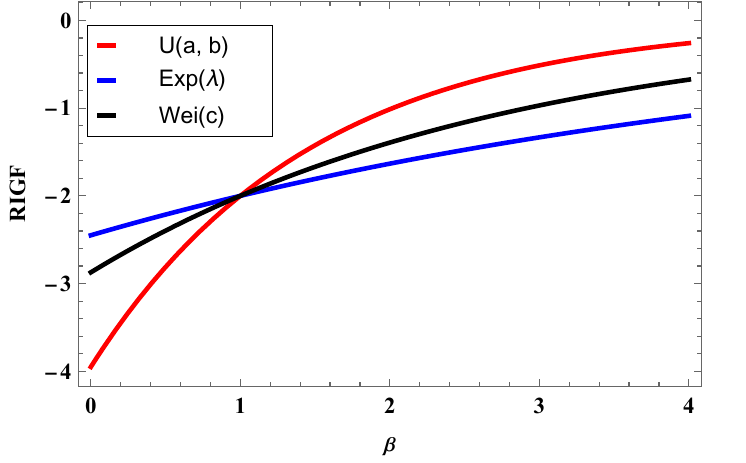}}
		\caption{Plots of the RIGFs of uniform distribution ($U(a,b)$) for $x\in[0.1,4]$, exponential distribution (Exp($\lambda$)) for $\lambda=1.5$, and Weibull (Wei($c$)) distribution for $c=1.4,$ $(a)$ for $\alpha=0.7$ and $(b)$ for $\alpha=1.5$ in Table \ref{tb2}.}
	\label{fig1}	
	\end{figure}

In the following proposition, we establish that the RIGF is shift independent, that is, it gives equal significance or weight to the occurrence of every event. We note that the shift-independent measures play a vital role in different fields, especially in information theory, pattern recognition, and signal processing. There is a chance of having time delays of signals in communication systems. In order to check the efficiency of such communication systems, the shift independent measure is crucial since it allows to measure the information conveyed in these signals without requiring precise alignment. This procedure helps one to understand data transmission in different networks in a better way. 
\begin{proposition}
Suppose the random variable $X$ has PDF $f(\cdot)$. Then, for $a>0$ and $b\geq0$, the RIGF of $Y=aX+b$ is
\begin{align}\label{eq2.7}
R^\alpha_\beta(Y)=a^{(1-\alpha)(\beta-1)}R^\alpha_\beta(X),~~\alpha>0,~~\alpha\ne 1,~~\beta>0.
\end{align}
\end{proposition}
\begin{proof}
 Under the assumption made, the PDF of $Y$ is obtained as $g(x)=\frac{1}{a}f(\frac{x-b}{a}),$ where $x\ge b$. Now, using this PDF in Definition \ref{def2.2}, the proof follows, and so it is omitted.  
\end{proof}

\begin{remark}
We note that some of the results presented here can be related with the properties of a variability measure in the sense of \cite{bickel2011descriptive}. For example, under suitable assumptions the following properties hold:
\begin{itemize}
\item if $X$ and $Y$ are equal in law, then $R^\alpha_\beta(X) = R^\alpha_\beta(Y)$;
\item $R^\alpha_\beta(X)>0$ for all $\beta>0$ and $0<\alpha<1;$
\item $R^\alpha_\beta(X+b)=R^\alpha_\beta(X)$  for all $b\ge 0$;
\item $R^\alpha_\beta(a X)=aR^\alpha_\beta(X)$  for all $a>0$, and $\alpha$ and $\beta$ such that $(1-\alpha)(\beta-1)=1$;
\item $X\le_{disp}Y$ implies $R^\alpha_\beta(X)\le R^\alpha_\beta(Y)$ (see Part $(A)$ of Proposition \ref{prop2.5}).
\end{itemize}
\end{remark}

In information theory, it is  always of interest to find a connection between a newly proposed information measure with other well-known information measures. In this regard, we next establish that the RIGF can be expressed in terms of the Shannon entropy of order $q>0$. We recall that for a continuous random variable $X$, the Shannon entropy of order $q$ is defined as (see \cite{kharazmi2021jensen})
\begin{align}\label{eq2.8}
\xi_{q}(X)=\int_{0}^{\infty}f(x)(-\log f(x))^qdx.
\end{align}

\begin{proposition}
Let $f(\cdot)$ be the PDF of a random variable $X$. Then, for $\beta\geq0$ and $0<\alpha<\infty, ~\alpha\neq1$, the RIGF of $X$ can be represented as
\begin{align}\label{eq2.9}
R^\alpha_\beta(X)=\delta(\alpha)\left(\sum_{q=0}^{\infty}\frac{(1-\alpha)^{q}}{q!}\xi_q(X)\right)^{\beta-1},
\end{align}
where $\xi_q(X)$ is as given in (\ref{eq2.8}).
\end{proposition} 
\begin{proof}
From (\ref{eq2.4}), we have 
\begin{align}\label{eq2.10}
R^\alpha_\beta(X)
&=\delta(\alpha)\left(E[e^{-(1-\alpha)\log f(X)}]\right)^{\beta-1}\nonumber\\
&=\delta(\alpha)\left(\sum_{q=0}^{\infty}\frac{(1-\alpha)^q}{q!}\int_{0}^{\infty}f(x)(-\log f(x))^qdx\right)^{\beta-1}.
\end{align}
From (\ref{eq2.10}), the result in (\ref{eq2.9}) follows directly, which completes the proof of the proposition.
\end{proof}

We now obtain upper and lower bounds for the RIGF. We recall that the bounds are useful to treat them as estimates when the actual form of the RIGF for a distributions is difficult to derive. 
\begin{proposition}\label{prop2.3}
Suppose $X$ is a continuous random variable with PDF $f(\cdot)$. Then,
\begin{itemize}
\item[$(A)$] for $0<\alpha<1$, we have
\begin{equation}\label{eq3.6*}
		R^\alpha_\beta(X)\left\{
		\begin{array}{ll}
			 \leq \delta(\alpha)G_{\alpha\beta-\alpha-\beta+2}(X),~	
			if~~0<\beta<1$ and $\beta\geq2,
			\\
			\geq\frac{1}{2}R^{\frac{\alpha+1}{2}}_{2\beta-1}(X),~~~~~~~~~~~if~ \beta\geq1,
			\\
			\leq\frac{1}{2}R^{\frac{\alpha+1}{2}}_{2\beta-1}(X),~~~~~~~~~~~if~ 0<\beta<1;
		\end{array}
		\right.
	\end{equation}
	\item[$(B)$] for $\alpha>1$, we have
\begin{equation}\label{eq3.7*}
	R^\alpha_\beta(X)\left\{
		\begin{array}{ll}
			 \leq \delta(\alpha)G_{\alpha\beta-\alpha-\beta+2}(X),~	
			if~~1<\beta<2,
			\\
			\leq\frac{1}{2}R^{\frac{\alpha+1}{2}}_{2\beta-1}(X),~~~~~~~~~~if~ \beta\geq1,
			\\
			\geq\frac{1}{2}R^{\frac{\alpha+1}{2}}_{2\beta-1}(X),~~~~~~~~~~if~ 0<\beta<1,
		\end{array}
		\right.
	\end{equation}
\end{itemize}
where $G_{\alpha\beta-\alpha-\beta+2}(X)=\int_{0}^{\infty}f^{\alpha\beta-\alpha-\beta+2}(x)dx$ is the IGF of $X$.
\end{proposition}
\begin{proof}
$(A)$ Let $\alpha\in(0,1)$. Consider a positive real-valued function $g(\cdot)$ such that $\int_{0}^{\infty}g(x)dx=1$. Then, the generalized Jensen inequality for a convex function $\psi(\cdot)$ is given by
\begin{align}\label{eq2.14}
	\psi\left(\int_{0}^{\infty}h(x)g(x)dx\right)\leq\int_{0}^{\infty}\psi(h(x))g(x)dx,
\end{align}
where $h(\cdot)$ is a real-valued function. Set $g(x)=f(x)$, $\psi(x)=x^{\beta-1}$ and $h(x)=f^{\alpha-1}(x)$. For $0<\beta<1$ and $\beta\geq2$, the function $\psi(x)$ is convex with respect to $x$. Thus, from (\ref{eq2.14}), we have 
\begin{align}\label{eq2.15}
	&\delta(\alpha)\left(\int_{0}^{\infty}f^\alpha(x)dx\right)^{\beta-1}\leq \delta(\alpha)\int_{0}^{\infty}f^{\alpha\beta-\alpha-\beta+2}(x)dx
	\Rightarrow R^\alpha_\beta(X)\leq \delta(\alpha)G_{\alpha\beta-\alpha-\beta+2}(X),
\end{align}
which establishes the first inequality in (\ref{eq3.6*}). 

In order to establish the second and third inequalities in (\ref{eq3.6*}), we require the Cauchy-Schwartz inequality. It is well-known that, for two real integrable functions $h_1(x)$ and $h_2(x)$, the Cauchy-Schwartz inequality is given by
\begin{align}\label{eq2.12}
\left(\int_{0}^{\infty}h_1(x)h_2(x)dx\right)^2\leq \int_{0}^{\infty}h^2_1(x)dx\int_{0}^{\infty}h^2_2(x)dx.
\end{align}
Taking $h_1(x)=f^{\frac{\alpha}{2}}(x)$ and $h_2(x)=f^{\frac{1}{2}}(x)$ in (\ref{eq2.12}), we obtain
\begin{align}\label{eq2.13}
&\left(\int_{0}^{\infty}f^{\frac{\alpha+1}{2}}(x)dx\right)^2\leq \int_{0}^{\infty}f^\alpha(x)dx.
\end{align}
Now, from (\ref{eq2.13}), we have for $\beta\geq1$,
\begin{align}\label{eq2.16*}
	\frac{1}{2(1-\frac{\alpha+1}{2})}\left(\int_{0}^{\infty}f^{\frac{\alpha+1}{2}}(x)dx\right)^{2(\beta-1)}\leq \delta(\alpha)\left(\int_{0}^{\infty}f^\alpha(x)dx\right)^{\beta-1},
\end{align}
and for $0<\beta<1$,
\begin{align}\label{eq2.15*}
\frac{1}{2(1-\frac{\alpha+1}{2})}\left(\int_{0}^{\infty}f^{\frac{\alpha+1}{2}}(x)dx\right)^{2(\beta-1)}\geq \delta(\alpha)\left(\int_{0}^{\infty}f^\alpha(x)dx\right)^{\beta-1}.
\end{align}
The second and third inequalities in (\ref{eq3.6*}) now follow from (\ref{eq2.16*}) and (\ref{eq2.15*}), respectively. 

The proof of Part $(B)$ for $\alpha>1$ is similar to the proof of Part $(A)$ for different values of $\beta$. So, the proof is omitted for brevity.  
\end{proof}

We now present an example to validate the result  in Proposition \ref{prop2.3}.

\begin{example}\label{ex2.1}
Suppose $X$ has an exponential distribution with PDF $f(x)=\lambda e^{-\lambda x},~ x\ge 0,~\lambda>0$. Then, 
\begin{align*}
R^\alpha_\beta(X)=\delta(\alpha)\left(\frac{\lambda^{\alpha-1}}{\alpha}\right)^{\beta-1},~~
R^{\frac{\alpha+1}{2}}_{2\beta-1}(X)=\delta(\alpha)\left(\frac{2\lambda^{\frac{\alpha-1}{2}}}{1+\alpha}\right)^{2(\beta-1)},~\text{and}~
G_{l}(X)=\frac{\lambda^{l-1}}{l},
\end{align*}
where $l=\alpha\beta-\alpha-\beta+2.$ In order to check the first two inequalities in (\ref{eq3.7*}), we have plotted the graphs of $R^\alpha_\beta(X)$, $\frac{1}{2}R^{\frac{\alpha+1}{2}}_{2\beta-1}(X)$ and $\delta(\alpha)G_{l}(X)$ in Figure $2$ for some choices of $\lambda$, $\beta,$ and $\alpha$.

\begin{figure}[htbp!]
  	\centering
  	\subfigure[]{\label{c1}\includegraphics[height=2in]{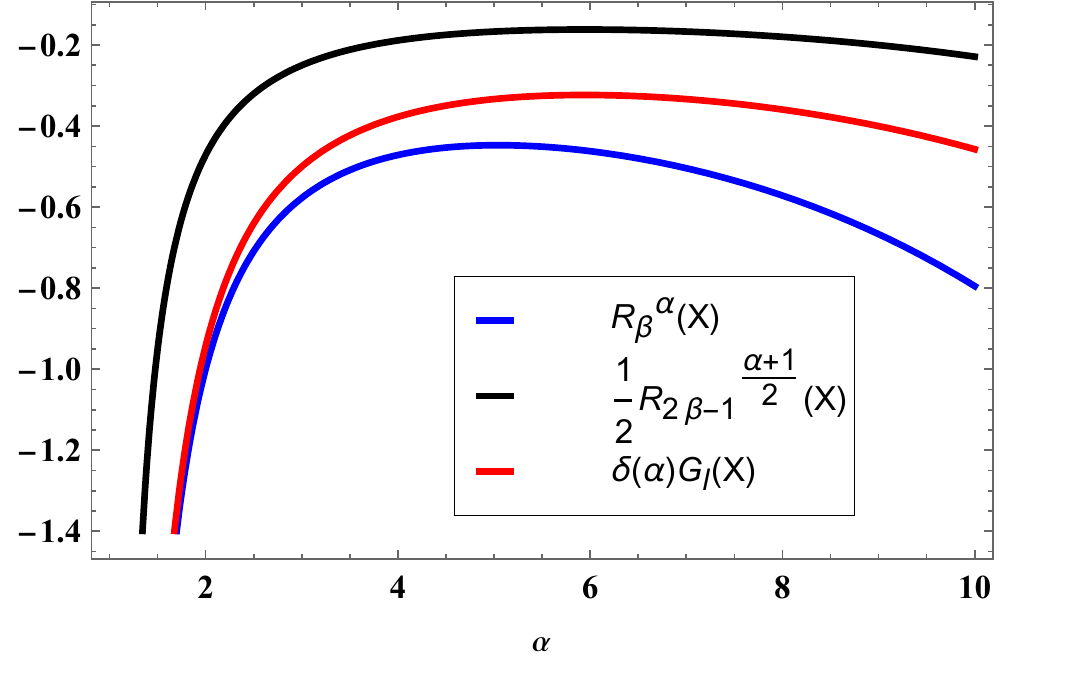}}
  	\subfigure[]{\label{c1}\includegraphics[height=2in]{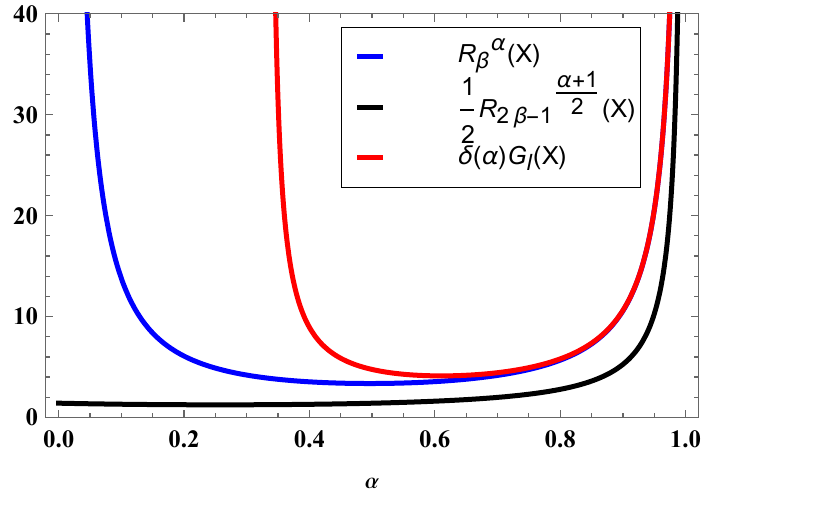}}
  	\caption{Graphs of $R^\alpha_\beta(X)$, $\frac{1}{2}R^{\frac{\alpha+1}{2}}_{2\beta-1}(X)$ and $\delta(\alpha)G_{l}(X),$ for $(a)$ $\lambda=2,~\beta=1.5$, and $\alpha>1$ and $(b)$  $\lambda=2,~\beta=2.5$ and $\alpha<1$ in Example \ref{ex2.1}.}
  \end{figure}
\end{example}

%

Suppose $X$ and $Y$  have  PDFs $f(\cdot)$ and $g(\cdot)$, respectively. The PDF of the sum of $X$ and $Y$, say $Z=X+Y$, is
\begin{align*}
f_{Z}(z)=\int_{0}^{z}f(x)g(z-x)dx,~~z\ge 0.
\end{align*}
This is known as the convolution of $X$ and $Y$. The convolution property is essential in various fields, particularly in signal processing, image processing, and deep learning. Convolution is used to filter signals, extract features, and perform operations like smoothing and differentiation in signal processing.
In image processing, it is fundamental in operations like blurring, sharpening, and edge detection. Here, we study the RIGF for convolution of two random variables $X$ and $Y$.

\begin{proposition}\label{pro2.4}
Let $f(\cdot)$ and $g(\cdot)$ be the PDFs of independent random variables $X$ and $Y$, respectively. Further, let $Z=X+Y.$ Then, for $0<\alpha<\infty$, $\alpha\neq1$,
\begin{itemize}
\item[$(A)$] $R^\alpha_\beta(Z)\leq R^\alpha_\beta(X)(G_\alpha(Y))^{\beta-1}$, if $0<\beta<1$;
\item[$(B)$] $R^\alpha_\beta(Z)\geq R^\alpha_\beta(X)(G_\alpha(Y))^{\beta-1}$, if $\beta\geq1$,
\end{itemize}
where $G_\alpha(Y)$ is the IGF of $Y$.
\end{proposition}
\begin{proof}
$(A)$ Case I: Consider $0<\beta<1$ and $0<\alpha<1$. From (\ref{eq2.4}), applying Jensen's inequality and Fubuni's theorem, we obtain
\begin{align}\label{eq2.17}
\int_{0}^{\infty}f^\alpha_Z(z)dz&=\int_{0}^{\infty}\left(\int_{0}^{z}f(x)g(z-x)dx\right)^\alpha dz\nonumber\\
&\geq \int_{0}^{\infty}\left(\int_{0}^{z}f^\alpha(x)g^\alpha(z-x)dx\right)dz\nonumber\\
&=\int_{0}^{\infty}f^\alpha(x)\left(\int_{x}^{\infty}g^\alpha(z-x)dz\right)dx\nonumber\\
\implies\delta(\alpha)\left(\int_{0}^{\infty}f^\alpha_Z(z)dz\right)^{\beta-1}&\leq\delta(\alpha)\left(\int_{0}^{\infty}f^\alpha(x)\left(\int_{x}^{\infty}g^\alpha(z-x)dz\right)dx\right)^{\beta-1}.
\end{align}
Case II: Consider $0<\beta<1$ and $\alpha>1$. Here, the proof follows similarly to the Case I. 
Thus, the result in Part $(A)$ is proved. 

The proof of Part $(B)$ is similar to that of Part $(A)$, and is therefore omitted.
\end{proof}

The following corollary is immediate from Proposition \ref{pro2.4}.
\begin{corollary}
For independent and identically distributed random variables $X$ and $Y$, with $0<\alpha<\infty$, $\alpha\neq1$, we have
\begin{itemize}
\item[$(A)$] $R^\alpha_\beta(Z)\leq R^\alpha_{2\beta-1}(X)$, if $0<\beta<1$,
\item[$(B)$] $R^\alpha_\beta(Z)\geq R^\alpha_{2\beta-1}(X)$, if $\beta\geq1$.
\end{itemize}
\end{corollary}

Numerous fields have benefited from the usefulness of the concept of stochastic orderings, including actuarial science, survival analysis, finance, risk theory, non-parametric approaches, and reliability theory. Suppose $X$ and $Y$ are two random variables with corresponding PDFs $f(\cdot)$ and $g(\cdot)$ and CDFs $F(\cdot)$ and $G(\cdot)$, respectively. Then, $X$ is less dispersed than $Y$, denoted by $X\leq_{disp}Y$, if $g(G^{-1}(x))\leq f(F^{-1}(x))$, for all $x\in(0,1)$. Further, $X$ is said to be smaller than $Y$ in the sense of the usual stochastic order (denote by $X\le_{st}Y$) if $F(x)\ge G(x)$, for $x>0$. For more details, one may refer to \cite{shaked2007stochastic}. 


The quantile representation of the RIGF of $X$ is given by
\begin{eqnarray*}
	R^{\alpha}_{\beta}(X)=\delta(\alpha)\left(\int_{0}^{1}f^{\alpha-1}(F^{-1}(u))du\right)^{{\beta-1}}.
\end{eqnarray*}

The next proposition deals with the comparisons of RIGFs of two random variables. The sufficient conditions here depend on the dispersive order and some restrictions of the parameters.

\begin{proposition}\label{prop2.5}
Consider two random variables $X$ and $Y$ such that $X\leq_{disp}Y$. Then, we have
\begin{itemize}
\item[$(A)$]$R^{\alpha}_{\beta}(X)\leq R^{\alpha}_{\beta}(Y)$; for $\{\alpha<1;\beta\geq1\}$ or $\{\alpha>1; \beta\ge1\},$
\item[$(B)$]$R^{\alpha}_{\beta}(X)\geq R^{\alpha}_{\beta}(Y)$, for $\{\alpha<1;\beta<1\}$ or $\{\alpha>1; \beta<1\}$.
\end{itemize}
\end{proposition}

\begin{proof}
$(A)$ Consider the case $\{\alpha<1;\beta\geq1\}$. The proof for the case $\{\alpha>1; \beta\ge1\}$ is quite similar. Under the assumption made, we have
\begin{align}\label{eq2.18}
X\leq_{disp}Y\implies f(F^{-1}(u))\geq g(G^{-1}(u))\implies f^{\alpha-1}(F^{-1}(u))\leq g^{\alpha-1}(G^{-1}(u))
\end{align}
for all $u\in(0,1).$ Thus, from (\ref{eq2.18}), we have 
\begin{align*}\label{eq2.19}
&\int_{0}^{1}f^{\alpha-1}(F^{-1}(u))du\leq\int_{0}^{1}g^{\alpha-1}(G^{-1}(u))du\nonumber\\
\implies&\delta(\alpha)\left(\int_{0}^{1}f^{\alpha-1}(F^{-1}(u))du\right)^{\beta-1}\leq \delta(\alpha)\left(\int_{0}^{1}g^{\alpha-1}(G^{-1}(u))du\right)^{\beta-1},
\end{align*}
establishing the required result. The proof for Part $(B)$ is similar, and is therefore omitted.
\end{proof}

Let $X$ be a random variable with CDF $F(\cdot)$ and quantile function $Q_{X}(u)$, for $0<u<1,$ given by 
\begin{eqnarray*}
	Q_{X}(u)=F^{-1}(u)=\inf\{x:F(x)\ge u\},~u\in(0,1).
\end{eqnarray*}
It is well-known that 
$X\le_{st} Y  \Longleftrightarrow Q_{X}(u)\le Q_{Y}(u),~u\in(0,1),$
where $Q_{Y}(\cdot)$ is the quantile function of $Y.$ Moreover, we know that if $X$ and $Y$ are such that they have a common finite left end point of their supports, then $X\le_{disp}Y\Rightarrow X\le_{st} Y$ (see \cite{shaked2007stochastic}). Next, we consider a convex and increasing function $\psi(\cdot)$, and then obtain inequalities between the RIGFs of $\psi(X)$ and $\psi(Y).$
\begin{proposition}
For the random variables $X$ and $Y$, with $X\le_{disp}Y,$ 
let $\psi(\cdot)$ be convex and strictly increasing. Then, we have 
\begin{equation}\label{eq2.25*}
		R^\alpha_\beta(\psi(X))\left\{
		\begin{array}{ll}
			 \ge R^\alpha_\beta(\psi(Y)),~	
			for~\{\alpha>1,\beta\le1\}~or~\{\alpha<1,\beta\ge1\},
			\\
			\le R^\alpha_\beta(\psi(Y)),~for~\{\alpha>1,\beta\ge1\}~or~\{\alpha<1,\beta\le1\}.
	 \end{array}
		\right.
	\end{equation}
\end{proposition}

\begin{proof}
Using the PDF of $\psi(X)$, the RIGF of $\psi(X)$ can be expressed as 
\begin{eqnarray*}
	R_{\beta}^{\alpha}(\psi(X))=\delta(\alpha)\left(\int_{0}^{1}\frac{f^{\alpha-1}(F^{-1}(u))}{(\psi^{\prime}(F^{-1}(u)))^{\alpha-1}}dx\right)^{\beta-1}.
\end{eqnarray*}
Since $\psi(\cdot)$ is assumed to be convex and increasing, with the assumption that $X\le_{disp}Y$, we obtain 
\begin{eqnarray*}
	\frac{f(F^{-1}(u))}{\psi'(F^{-1}(u))}\ge \frac{g(G^{-1}(u))}{\psi'(G^{-1}(u))}.
\end{eqnarray*} 	
Now, using 	$\alpha>1$ and $\beta\le1,$ the first inequality in (\ref{eq2.25*}) follows easily. The inequalities for other restrictions on $\alpha$ and $\beta$ can be established similarly. This completes the proof of the proposition.  
\end{proof}

Escort distributions are useful in modelling and analysing complex systems, where traditional probabilistic models fail. They provide a flexible and robust framework for dealing with non-standard distributions, making them essential in many areas of research and applications. Escort distributions are also used for the characteristic of chaos and multifractals in statistical physics. \cite{abe2003geometry}  showed quantitatively that it is inappropriate to use the original
distribution instead of the escort distribution for calculating the expectation values of physical quantities in nonextensive statistical mechanics. Suppose $X$ and $Y$ are two continuous random variables and their PDFs  are $f(\cdot)$ and $g(\cdot)$, respectively. Then, the PDFs of the escort and generalized escort distributions are, respectively, given by 
\begin{align}\label{eq2.21}
f_{e,r}(x)=\frac{f^r(x)}{\int_{0}^{\infty}f^r(x)dx},~ x>0,~\mbox{and}~~g_{E,r}(x)=\frac{f^r(x)g^{1-r}(x)}{\int_{0}^{\infty}f^r(x)g^{1-r}(x)dx},~x>0.
\end{align}

In the following proposition, we express the RIGF of the escort distribution in terms of the RIGF of baseline distribution. The result follows directly from (\ref{eq2.4}) and (\ref{eq2.21}).

\begin{proposition}
Let $X$ be a continuous random variable with PDF $f(\cdot)$. Then, the RIGF of the escort random variable of order $r$ can be obtained as
\begin{align*}
R^\alpha_\beta(X_{e,r})=\frac{(1-\alpha r)}{(1-\alpha)(1-r)}\times \frac{R^{\alpha r}_\beta(X)}{R^r_{\alpha\beta-\alpha+1}(X)},
\end{align*}
where $X_{e,r}$ is the escort random variable.
\end{proposition}

\section{R\'enyi divergence information generating function}\label{sec3}
We propose an information generating function of the R\'enyi divergence. Suppose $X$ and $Y$ are two continuous random variables and their PDFs are $f(\cdot)$ and $g(\cdot),$ respectively. Then, the R\'enyi divergence information generating function (RDIGF) is given by 
\begin{align}\label{eq3.1}
RD^\alpha_\beta(X,Y)=\delta^*(\alpha)\left(\int_{0}^{\infty}\left(\frac{f(x)}{g(x)}\right)^\alpha g(x)dx\right)^{\beta-1}=\delta^*(\alpha)\left(E_g\bigg[\frac{f(X)}{g(X)}\bigg]^\alpha\right)^{\beta-1}.
\end{align}
Clearly, the integral in (\ref{eq3.1}) exists for $0<\alpha<\infty$ and $\beta>0$. 
The $k$th order derivative of (\ref{eq3.1}) with respect to $\beta$ is obtained as
\begin{align}\label{eq3.2}
\frac{\partial RD^\alpha_\beta(X,Y)}{\partial \beta^k}=\delta^*(\alpha)\left(\int_{0}^{\infty}\left(\frac{f(x)}{g(x)}\right)^\alpha g(x)dx\right)^{\beta-1}\left(\log\int_{0}^{\infty}\left(\frac{f(x)}{g(x)}\right)^\alpha g(x)dx\right)^k,
\end{align}
provided  the integral exists. The following observations from (\ref{eq3.1}) and (\ref{eq3.2}) can be readily made:
\begin{itemize}
\item $RD^\alpha_\beta(X,Y)|_{\beta=1}=\delta^*(\alpha)$;
 $\frac{\partial}{\partial \beta}RD^\alpha_\beta(X,Y)|_{\beta=1}=RD(X,Y)$;
\item $RD^\alpha_\beta(X,Y)=\alpha \delta^*(\alpha)RD^{1-\alpha}_\beta(Y,X)$,
\end{itemize} 
where $RD(X,Y)$ is the R\'enyi divergence between $X$ and $Y$ given in (\ref{eq1.2}). In Table $3$, we present closed-form expressions of the RDIGF and R\'enyi divergence for some continuous distributions.  In addition, to check the behaviour of the RDIGFs in Table \ref{tb3}, we plot them in Figure \ref{fig3}. We notice that the RDIGFs are increasing with respect to $\beta>0.$

\begin{table}[h!]
\caption {The RDIGF and R\'enyi divergence for Pareto type-I, exponential, and Lomax distributions.}
	\centering 
	\scalebox{.77}{\begin{tabular}{c c c c c c c c } 
			\hline\hline\vspace{.1cm} 
			PDFs  & RDIGF & R\'enyi divergence \\
			\hline
				$f(x)=c_1x^{-(c_1+1)},~ g(x)=c_2x^{-(c_2+1)},~x>1,c_1,c_2>0$&$\delta^*(\alpha)\left(\frac{c^\alpha_1 c^{1-\alpha}_2}{\alpha c_1+(1-\alpha)c_2}\right)^{\beta-1}$ &$\delta^*(\alpha)\log \left(\frac{c^\alpha_1 c^{1-\alpha}_2}{\alpha c_1+(1-\alpha)c_2}\right)$	\\	[2EX]
				$f(x)=\lambda_1e^{-\lambda_1x},~ g(x)=\lambda_2e^{-\lambda_2x},~x>0,~\lambda_1,\lambda_2>0$& $\delta^*(\alpha)\left(\frac{\lambda_1^\alpha\lambda_2^{1-\alpha}}{(\alpha-1)\lambda_2-\alpha\lambda_1}\right)^{\beta-1}$ & $\delta^*(\alpha)\log\left(\frac{\lambda_1^\alpha\lambda_2^{1-\alpha}}{(\alpha-1)\lambda_2-\alpha\lambda_1}\right)$\\[2EX]
$f(x)=\frac{b_1}{a}(1+\frac{x}{a})^{-(b_1+1)},~ g(x)=\frac{b_2}{a}(1+\frac{x}{a})^{-(b_2+1)},~x>0,~a,b_1,b_2>0$& $\delta^*(\alpha)\left(\frac{b_1^\alpha b_2^{1-\alpha}}{\alpha(b_1-b_2)+b_2}\right)^{\beta-1}$&$\delta^*(\alpha)\log\left(\frac{b_1^\alpha b_2^{1-\alpha}}{\alpha(b_1-b_2)+b_2}\right)$\\[1EX]
	\hline	 		
	\end{tabular}} 
	\label{tb3} 
\end{table}

\begin{figure}[h!]
		\centering
	\subfigure[]{\label{c1}\includegraphics[height=2in]{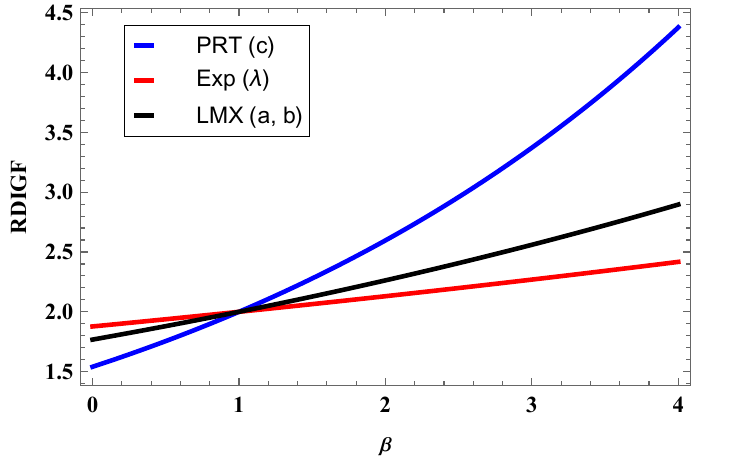}}
\subfigure[]{\label{c1}\includegraphics[height=2in]{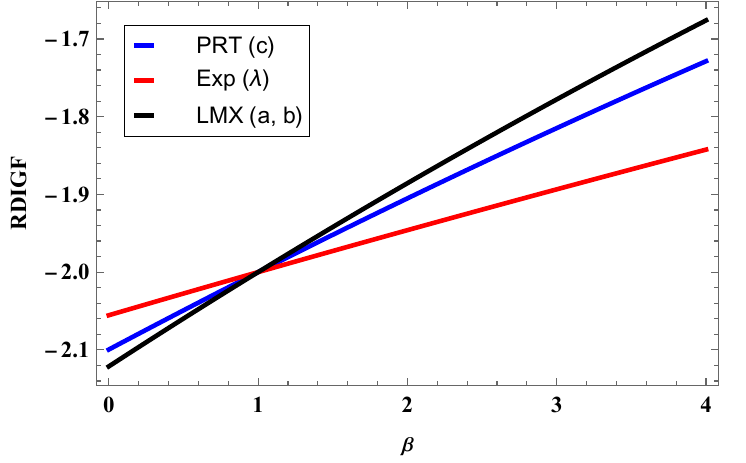}}
		\caption{Plots of the RDIGFs of Pareto type-I (PRT) with $c_1=0.8,~c_2=1.5$, exponential (Exp) with $\lambda_1=0.8,~\lambda_2=0.5,$ and Lomax  (LMX) distributions with $a=0.5,~b_1=0.8,$ and $b_2=0.4$ when $(a)$ $\alpha=0.5$ and $(b)$ $\alpha=1.5 $.}
	\label{fig3}	
	\end{figure}

The following proposition states that the RDIGF between two random variables $X$ and $Y$ becomes the RIGF of $X$ if $Y$ follows uniform distribution in $[0,1]$. The proof here is omitted since it is straightforward. 
\begin{proposition}
Let $X$ be a continuous random variable and $Y$ be an uniform random variable, i.e. $Y\sim U(0,1)$. Then, the RDIGF of $X$ reduces to the RIGF of $X$.
\end{proposition}

Next, we establish a relation between the RIGF and RDIGF.  In this regard, we consider the generalized escort distribution with PDF as in (\ref{eq2.21}).
\begin{proposition}
Let $Y_{e,r}$, $X_{e,r}$ be the escort random variables and $Y_{E,r}$ be the generalized escort random variable. Then,
\begin{align*}
R^\alpha_\beta(Y_{E,r})RD^r_{\alpha\beta-\alpha+1}(X,Y)=(1-\alpha)R^\alpha_{r\beta-r+1}(X)R^\alpha_{(1-r)(\beta-1)+1}(Y)RD^r_\beta(X_{e,\alpha},Y_{e,\alpha}).
\end{align*}
\end{proposition}
\begin{proof}
Using (\ref{eq2.4}) and (\ref{eq2.21}), we obtain
\begin{align}\label{eq2.25}
R^\alpha_\beta(Y_{E,r})&=\delta(\alpha)\left(\int_{0}^{\infty}\frac{f^{\alpha r}(x)g^{\alpha(1-r)}(x)dx}{\left(\int_{0}^{\infty}f^r(x)g^{(1-r)}(x)dx\right)^\alpha}\right)^{\beta-1}\nonumber\\
&=\delta(\alpha)\frac{\left(\int_{0}^{\infty}f^{\alpha r}(x)g^{\alpha(1-r)}(x)dx\right)^{\beta-1}}{\left(\int_{0}^{\infty}f^r(x)g^{(1-r)}(x)dx\right)^{\alpha(\beta-1)}}\nonumber\\
&=\delta(\alpha)\frac{\left(\int_{0}^{\infty}f^{\alpha r}(x)g^{\alpha(1-r)}(x)dx\right)^{\beta-1}}{(r-1)RD^r_{\alpha\beta-\alpha+1}(X,Y)}\nonumber\\
&=\delta(\alpha)\frac{\left(\int_{0}^{\infty}(\frac{f^{\alpha}(x)}{\int_{0}^{\infty}f^{\alpha}(x)dx})^r(\frac{g^{\alpha}(x)}{\int_{0}^{\infty}g^{\alpha}(x)dx})^{1-r}dx\right)^{\beta-1}\{(\int_{0}^{\infty}f^{\alpha}(x)dx)^{r}(\int_{0}^{\infty}g^{\alpha}(x)dx)^{1-r}\}^{\beta-1}}{(r-1)RD^r_{\alpha\beta-\alpha+1}(X,Y)}.
\end{align}

Now, the required result follows easily from (\ref{eq2.25}). 
\end{proof}

Monotone functions are fundamental in many theoretical and practical applications due to their predictability, order-preserving nature and the mathematical simplicity they bring to various problems. In optimization problems, monotone functions are particularly useful because they simplify the process of finding maximum or minimum values. In statistics, monotone likelihood ratios are used in hypothesis testing and decision theory, where the monotonicity of certain functions ensures the validity of statistical tests and models. In the following, we discuss the effect of the RDIGF for monotone transformations.

\begin{proposition}
Suppose $f(\cdot)$ and  $g(\cdot)$ are the PDFs of $X$ and $Y,$ respectively, and $\psi(\cdot)$ is strictly monotonic, differential, and invertible function. Then,
\begin{equation*}\label{eq3.6}
		RD^\alpha_\beta(\psi(X),\psi(Y))=\left\{
		\begin{array}{ll}
			 RD^\alpha_\beta(X,Y),~~~	
			if~\psi~ is ~strictly~ increasing,
			\\
			\\
			-RD^\alpha_\beta(X,Y),~if~ \psi~ is ~strictly~ decreasing.
		\end{array}
		\right.
	\end{equation*}
\end{proposition}
\begin{proof}
The PDFs of $\psi(X)$ and $\psi(Y)$ are $$f_{\psi}(x)=\frac{1}{|\psi^{'}(\psi^{-1}(x))|}f(\psi^{-1}(x)) ~~\text{and}~~g_{\psi}(x)=\frac{1}{|\psi^{'}(\psi^{-1}(x))|}g(\psi^{-1}(x)),~~ x\in\big(\psi(0),\psi(\infty)\big),$$
respectively.
Let us first consider  $\psi(\cdot)$ to be strictly increasing. From  (\ref{eq3.1}), we have
\begin{align*}
RD^\alpha_\beta(\psi(X),\psi(Y))&=\delta^*(\alpha)\left(\int_{\psi(0)}^{\psi(\infty)}f^\alpha_\psi(x)g^{1-\alpha}_\psi(x)dx\right)^{\beta-1}\nonumber\\
&=\delta^*(\alpha)\left(\int_{\psi(0)}^{\psi(\infty)}\frac{f^\alpha(\psi^{-1}(x))g^{1-\alpha}(\psi^{-1}(x))}{\psi^{'}(\psi^{-1}(x))}dx\right)^{\beta-1}\nonumber\\
&=\delta^*(\alpha)\left(\int_{0}^{\infty}f^\alpha(x)g^{1-\alpha}(x)dx\right)^{\beta-1}.
\end{align*}
Hence, $RD^\alpha_\beta(\psi(X),\psi(Y))=RD^\alpha_\beta(X,Y)$. 
We can similarly prove the result for strictly decreasing function $\psi(\cdot)$. This completes the proof of the proposition.
\end{proof}

\section{Estimation of the RIGF}\label{sec5}
In this section, we discuss some non-parametric and parametric estimators of the RIGF. A Monte Carlo simulation study is then carried out for the comparison of these two estimators. A real data set is also analyzed for illustrative purpose.

\subsection{Non-parametric estimator of the RIGF}
We first propose a non-parametric estimator of the RIGF in (\ref{eq2.4}) based on kernel estimator. Denote by $\widehat f(\cdot)$ the kernel estimator of the PDF $f(\cdot)$, given by
 \begin{eqnarray}\label{eq5.1}
 \widehat f(x)=\frac{1}{n\beta_n}\sum_{i=1}^{n}J\left(\frac{x-X_i}{\beta_n}\right), 
 \end{eqnarray} 
where $J(\cdot)~(\ge0)$ is known as kernel and $\{\beta_n\}$ is a sequence of real numbers, known as bandwidths, satisfying $\beta_n\rightarrow0$ and $n\beta_n\rightarrow0$ for $n\rightarrow0.$ For more details, see \cite{rosenblatt1956remarks} and \cite{parzen1962estimation}. Note that the kernel $J(\cdot)$ satisfies the following properties:
\begin{itemize}
\item[$(a)$] It is non-negative, i.e. $J(x)\ge0;$
\item[$(b)$] $\int J(x)dx=1;$
\item[$(c)$] It is symmetric about zero;
\item[$(d)$] $J(\cdot)$ satisfies the Lipschitz condition.
\end{itemize}
Thus, based on the kernel, a non-parametric kernel estimator of the RIGF in (\ref{eq2.4}) is defined as
\begin{align}\label{eq5.2}
\widehat {R}^\alpha_\beta(X)=\delta(\alpha)\left(\int_{0}^{\infty}\widehat f^\alpha(x)dx\right)^{\beta-1},~~0<\alpha<\infty,~ \alpha\neq1,~ \beta>0.
\end{align}
Further,  a non-parametric kernel estimator of the IGF given in (\ref{eq1.3}) is obtained as
\begin{align}\label{eq5.3*}
\widehat {G}_\alpha(X)=\int_{0}^{\infty}\widehat f^\alpha(x)dx,~~\alpha>0.
\end{align}

Next, we carry out a Monte Carlo simulation study to examine the performance of the non-parametric estimators of the RIGF and IGF given in (\ref{eq5.2}) and (\ref{eq5.3*}), respectively. 
We use Monte Carlo simulation  to generate data  from Weibull distribution with shape parameter $k>0$ and scale parameter $\lambda>0$ for different sample sizes. The SD, AB, and MSE of the kernel based non-parametric estimators of the RIGF in (\ref{eq5.2}) and IGF in (\ref{eq5.3*}) are then obtained based on $500$ replications. Here, we have employed Gaussian kernel, given by
\begin{align}\label{eq5.3}
k(z)=\frac{1}{\sqrt{2\pi}}e^{-\frac{z^2}{2}},~-\infty<z<\infty.
\end{align}  
The SD, AB, and MSE of the non-parametric estimators $\widehat {R}^\alpha_\beta(X)$ and $\widehat {G}_\alpha(X)$ are then computed and are presented for different choices of $n,k,\lambda, \alpha$, and $\beta$ in Tables \ref{tb4} and  \ref{tb5*}. The software ``Mathematica" has been used for simulational purpose. From Tables \ref{tb4} and \ref{tb5*}, we observe the following:
\begin{itemize}
\item The SD, AB, and MSE decrease as the sample size $n$ increases, verifying the consistency of the proposed estimators;
\item The non-parametric estimator of the RIGF performs better than that of the IGF in terms of the SD, AB, and MSE; 

\end{itemize}

\begin{table}[ht!]
\caption {Comparison between  the non-parametric estimators of the IGF in (\ref{eq5.3*}) and RIGF in (\ref{eq5.2}) in terms of the AB, MSE, and  SD for different choices of $\alpha$, $\beta$, $k,~\lambda$, and $n$.}
	\begin{center}

	\scalebox{0.7}{\begin{tabular}{c c c c c c c c c c c c c c } 
			\hline\hline 
		\textbf{$\alpha$} & $\textbf{n}$ & \multicolumn{4}{c}{\bf{$\beta=1.1,~\lambda=1.5,~k=2$}}& \textbf{$\beta$}& $\textbf{n}$ & \multicolumn{4}{c}{$\alpha=0.3,~\lambda=1,~k=2$ }  \\
			\hline
			&&	\multicolumn{2}{c}{\textbf{IGF}}& \multicolumn{2}{c}{\textbf{RIGF}} &&&\multicolumn{2}{c}{\textbf{IGF}}& \multicolumn{2}{c}{\textbf{RIGF} } \\
						\hline
			&& 	 \textbf{SD}& \textbf{AB}& \textbf{SD} & \textbf{AB} &&& \textbf{SD} & \textbf{AB} & \textbf{SD} &\textbf{ AB}  \\
			 &~&~&\textbf{(MSE)} & ~& \textbf{(MSE)}&~&~ & ~ &\textbf{(MSE)} &~ & \textbf{(MSE)} \\
			\hline\hline
\multirow{10}{1.9cm}
~ & $150$ & $ 0.10189$ & $0.06812$& $0.00659$ &$0.00448$&~& 150 & $0.07575$& $0.04745$& $0.02484$& $0.01615$  \\[0.5ex]
~& ~& ~ & $(0.01502)$&~  & $(0.00006)$ &~&~&~  & $( 0.00799)$ &~  & $(0.00088)$   \\[1.2ex]
$0.3$ & $300$ & $0.07850$ & $0.05581$& $0.00505$ & $0.00364$ & 0.4 & 300& $0.05915$& $0.04524$& $0.01949$& $0.01510$   \\[0.5ex]
~& ~& ~ & $(0.00928)$&~  & $(0.00004)$ &~ &~&~ & $(0.00555)$ &~  & $(0.00061)$  \\[1.2ex]
~ & $500$ & $0.06867$ & $0.04703$& $0.00440$ &$0.00305$&~ & 500& $0.05283$& $0.03557$& $0.01723$& $0.01182$   \\[0.5ex]
~& ~& ~ & $(0.00693)$&~  & $(0.00003)$&~&~ &~  & $(0.00406)$ &~  & $( 0.00044)$  \\[1.2ex]

\hline

\multirow{10}{1.9cm}
~ & $150$ & $0.01249$ & $0.00861$& $0.00521$ &$0.00360$&~& 150 & $0.07737$& $0.04824$& $0.01927$& $0.01231$   \\[0.5ex]
~& ~& ~ & $(0.00023)$&~  & $(0.00004)$ &~&~&~  & $(0.00831)$ &~  & $(0.00052)$  \\[1.2ex]
$0.8$ & $300$ & $0.00865$ & $0.00506$& $0.00360$ & $0.00211$ & 0.6 & 300& $0.06291$& $0.04286$& $0.01544$& $0.01078$   \\[0.5ex]
~& ~& ~ & $(0.00010)$&~  & $(0.00002)$ &~ &~&~ & $(0.00579)$ &~  & $( 0.00035)$  \\[1.2ex]
~ & $500$ & $0.00678$ & $0.00344$& $0.00281$ &$0.00143$&~ & 500& $0.05248$& $0.04012$& $0.01289$& $0.00999$   \\[0.5ex]
~& ~& ~ & $( 0.00006)$&~  & $(0.00001)$&~&~ &~  & $(0.00436)$ &~  & $(0.00027)$  \\[1.2ex]

\hline

\multirow{10}{1.9cm}
~ & $150$ & $0.00916$ & $0.02136$& $0.00558$ &$0.01288$&~& 150 & $0.07915$& $0.04806$& $0.00587$& $0.00365$  \\[0.5ex]
~& ~& ~ & $(0.00054)$&~  & $(0.00020)$ &~&~&~  & $(0.00857)$ &~  & $(0.00005)$   \\[1.2ex]
$1.2$ & $300$ & $0.00655$ & $0.01697$& $0.00397$ & $0.01020$ & 0.9 & 300& $0.06017$& $0.04170$& $0.00443$& $0.00312$   \\[0.5ex]
~& ~& ~ & $(0.00033)$&~  & $(0.00012)$ &~ &~&~ & $( 0.00536)$ &~  & $(0.00003)$  \\[1.2ex]
~ & $500$ & $0.00495$ & $0.01447$& $0.00299$ &$0.00868$&~ & 500& $0.05430$& $0.03995$& $0.00400$& $0.00298$   \\[0.5ex]
~& ~& ~ & $( 0.00023)$&~  & $(0.00008)$&~&~ &~  & $(0.00454)$ &~  & $(0.00002)$  \\[1.2ex]

\hline

\multirow{10}{1.9cm}
~ & $150$ & $0.01601$ & $0.02631$& $0.00502$ &$0.00818$&~& 150 & $0.07676$& $0.04366$& $0.01357$& $0.00788$  \\[0.5ex]
~& ~& ~ & $(0.00095)$&~  & $(0.00009)$ &~&~&~  & $(0.00780)$ &~  & $(0.00025)$   \\[1.2ex]
$1.5$ & $300$ & $0.01083$ & $0.02100$& $0.00338$ & $0.00649$ & 1.2 & 300& $0.06111$& $0.03892$& $0.01074$& $0.00696$   \\[0.5ex]
~& ~& ~ & $(0.00056)$&~  & $(0.00005)$ &~ &~&~ & $(0.00525)$ &~  & $(0.00016)$  \\[1.2ex]
~ & $500$ & $0.00921$ & $0.01779$& $0.00286$ &$0.00548$&~ & 500& $0.05310$& $0.03829$& $0.00935$& $0.00681$   \\[0.5ex]
~& ~& ~ & $( 0.00040)$&~  & $(0.00004)$&~&~ &~  & $( 0.00429)$ &~  & $(0.00013)$  \\[1.2ex]

\hline

\multirow{10}{1.9cm}
~ & $150$ & $0.02045$ & $0.02765$& $0.00477$ &$ 0.00636$&~& 150 & $0.07740$& $0.04885$& $0.05220$ & $0.03325$  \\[0.5ex]
~& ~& ~ & $(0.00118)$&~  & $(0.00006)$ &~&~&~  & $(0.00838)$ &~  & $(0.00383)$   \\[1.2ex]
$2.0$ & $300$ & $0.01539$ & $0.02244$& $0.00353$ & $0.00511$ & 1.6 & 300& $0.06528$& $0.03730$& $0.04387$& $0.02533$   \\[0.5ex]
~& ~& ~ & $(0.00074)$&~  & $(0.00004)$ &~ &~&~ & $(0.00565)$ &~  & $(0.00257)$  \\[1.2ex]
~ & $500$ & $0.01145$ & $0.01959$& $0.00262$ &$0.00442$&~ & 500& $0.05503$& $0.03502$& $0.03704$& $0.02371$   \\[0.5ex]
~& ~& ~ & $( 0.00051)$&~  & $(0.00003)$&~&~ &~  & $(0.00425)$ &~  & $( 0.00193)$  \\[1.2ex]

\hline \hline
		\label{tb4} 		
	\end{tabular}} 
	\end{center}
	\end{table}

\begin{table}[ht!]
\caption {Continuation of Table \ref{tb4}.}
	\begin{center}

	\scalebox{0.7}{\begin{tabular}{c c c c c c c c c c c c c c } 
			\hline\hline 
		\textbf{$k$} & $\textbf{n}$ & \multicolumn{4}{c}{\bf{$\alpha=0.3,~\beta=0.5,~\lambda=1$}}& \textbf{$\lambda$}& $\textbf{n}$ & \multicolumn{4}{c}{$\alpha=0.3,~\beta=0.5,~k=2$ }  \\
			\hline
			&&	\multicolumn{2}{c}{\textbf{IGF}}& \multicolumn{2}{c}{\textbf{RIGF}} &&&\multicolumn{2}{c}{\textbf{IGF}}& \multicolumn{2}{c}{\textbf{RIGF} } \\
						\hline
			&& 	 \textbf{SD}& \textbf{AB}& \textbf{SD} & \textbf{AB} &&& \textbf{SD} & \textbf{AB} & \textbf{SD} &\textbf{ AB}  \\
			 &~&~&\textbf{(MSE)} & ~& \textbf{(MSE)}&~&~ & ~ &\textbf{(MSE)} &~ & \textbf{(MSE)} \\
			\hline\hline
\multirow{10}{1.9cm}
~ & $150$ & $ 0.98454$ & $4.33516$& $0.03814$ &$0.11774$&~& 150 & $0.07170$& $0.04462$& $0.02314$& $0.01500$  \\[0.5ex]
~& ~& ~ & $(19.7629)$&~  & $( 0.01532)$ &~&~&~  & $( 0.00713)$ &~  & $(0.00076)$   \\[1.2ex]
$0.5$ & $300$ & $ 0.78263$ & $3.95495$& $0.02788$ & $ 0.10216$ & 0.9 & 300& $0.05478$& $0.03718$& $0.01770$& $0.01227$   \\[0.5ex]
~& ~& ~ & $(16.2541)$&~  & $(0.01121)$ &~ &~&~ & $(0.00438)$ &~  & $(0.00046)$  \\[1.2ex]
~ & $500$ & $0.72230$ & $3.69864$& $0.02411$ &$0.09292$&~ & 500& $ 0.04830$& $0.03308$& $0.01552$& $ 0.01085$   \\[0.5ex]
~& ~& ~ & $(14.2016)$&~  & $( 0.00922)$&~&~ &~  & $( 0.00343)$ &~  & $( 0.00036)$  \\[1.2ex]

\hline

\multirow{10}{1.9cm}
~ & $150$ & $0.34114$ & $0.90743$& $0.03687$ &$0.08441$&~& 150 & $0.08687$& $0.05432$& $0.02098$& $0.01349$   \\[0.5ex]
~& ~& ~ & $(0.93981)$&~  & $(0.00848)$ &~&~&~  & $(0.01050)$ &~  & $( 0.00062)$  \\[1.2ex]
$0.8$ & $300$ & $0.26951$ & $0.80978$& $0.02772$ & $0.07305$ & 1.2 & 300& $0.07165$& $0.04238$& $0.01699$& $0.01041$   \\[0.5ex]
~& ~& ~ & $( 0.72838)$&~  & $(0.00610)$ &~ &~&~ & $(0.00693)$ &~  & $( 0.00040)$  \\[1.2ex]
~ & $500$ & $0.21772$ & $0.72045$& $0.02146$ &$ 0.06344$&~ & 500& $0.05675$& $0.03601$& $0.01347$& $0.00873$   \\[0.5ex]
~& ~& ~ & $(0.56645)$&~  & $( 0.00445)$&~&~ &~  & $(0.00452)$ &~  & $( 0.00026)$  \\[1.2ex]

\hline

\multirow{10}{1.9cm}
~ & $150$ & $0.23475$ & $0.47731$& $0.03457$ &$0.06498$&~& 150 & $0.10715$& $0.06777$& $0.01903$& $0.01245$  \\[0.5ex]
~& ~& ~ & $(0.28293)$&~  & $( 0.00542)$ &~&~&~  & $(0.01607)$ &~  & $(0.00052)$   \\[1.2ex]
$1.0$ & $300$ & $0.17896$ & $0.40077$& $0.02555$ & $0.05292$ & 1.6 & 300& $0.08041$& $0.05211$& $ 0.01414$& $0.00940$   \\[0.5ex]
~& ~& ~ & $(0.19264)$&~  & $( 0.00345)$ &~ &~&~ & $( 0.00918)$ &~  & $(0.00029)$  \\[1.2ex]
~ & $500$ & $0.15195$ & $0.38037$& $0.02145$ &$0.04969$&~ & 500& $0.07334$& $0.04845$& $0.01294$& $ 0.00870$   \\[0.5ex]
~& ~& ~ & $( 0.16777)$&~  & $(0.00293)$&~&~ &~  & $( 0.00773)$ &~  & $(0.00024)$  \\[1.2ex]

\hline

\multirow{10}{1.9cm}
~ & $150$ & $0.11089$ & $0.15138$& $0.02545$ &$0.03403$&~& 150 & $0.12280$& $0.07078$& $0.01725$& $0.01030$  \\[0.5ex]
~& ~& ~ & $(0.03521)$&~  & $( 0.00181)$ &~&~&~  & $(0.02009)$ &~  & $(0.00040)$   \\[1.2ex]
$1.5$ & $300$ & $0.09530$ & $0.12626$& $0.02152$ & $0.02802$ & 2.0 & 300& $0.10036$& $0.06678$& $0.01408$& $0.00956$   \\[0.5ex]
~& ~& ~ & $(0.02502)$&~  & $(0.00125)$ &~ &~&~ & $(0.01453)$ &~  & $( 0.00029)$  \\[1.2ex]
~ & $500$ & $0.07919$ & $0.12346$& $0.01780$ &$0.02717$&~ & 500& $0.08262$& $0.06355$& $0.01155$& $0.00899$   \\[0.5ex]
~& ~& ~ & $( 0.02151)$&~  & $(0.00106)$&~&~ &~  & $( 0.01086)$ &~  & $(0.00021)$  \\[1.2ex]

\hline

\multirow{10}{1.9cm}
~ & $150$ & $0.06089$ & $0.00067$& $0.02087$ &$0.00035$&~& 150 & $0.14479$& $0.08922$& $0.01608$& $0.01025$  \\[0.5ex]
~& ~& ~ & $( 0.00371)$&~  & $( 0.00044)$ &~&~&~  & $( 0.02892)$ &~  & $(0.00036)$   \\[1.2ex]
$2.5$ & $300$ & $0.04580$ & $0.00434$& $0.01573$ & $0.00183$ & 2.5 & 300& $0.11921$& $0.07276$& $0.01316$& $0.00826$   \\[0.5ex]
~& ~& ~ & $(0.00212)$&~  & $( 0.00025)$ &~ &~&~ & $( 0.01950)$ &~  & $(0.00024)$  \\[1.2ex]
~ & $500$ & $0.03822$ & $0.00096$& $0.01315$ &$0.00010$&~ & 500& $0.09865$& $0.07268$& $0.01090$& $0.00815$   \\[0.5ex]
~& ~& ~ & $( 0.00146)$&~  & $( 0.00017)$&~&~ &~  & $(0.01501)$ &~  & $( 0.00019)$  \\[1.2ex]

\hline \hline
		\label{tb5*} 		
	\end{tabular}} 
	\end{center}
	\end{table}

\subsection{Parametric estimator of the RIGF}
In the previous subsection, we examined the performance of the non-parametric estimators of both RIGF and IGF. Here, we will focus on the parametric estimation of the RIGF and IGF when the probability distribution is Weibull. For the Weibull distribution with shape parameter $k>0$ and scale parameter $\lambda>0$, the RIGF and IGF are, respectively, given by 
\begin{align}\label{eq5.4}
R^\alpha_\beta(X)=\delta(\alpha)\left(\int_{0}^{\infty}\Big\{\frac{k}{\lambda}\Big(\frac{x}{\lambda}\Big)^{k-1}e^{-(\frac{x}{\lambda})^k}\Big\}^\alpha dx\right)^{\beta-1},~~0<\alpha<\infty,~ \alpha\neq1,~ \beta>0,
\end{align} 
and 
  \begin{align}\label{eq5}
  G_\alpha(X)=\int_{0}^{\infty}\Big\{\frac{k}{\lambda}\Big(\frac{x}{\lambda}\Big)^{k-1}e^{-(\frac{x}{\lambda})^k}\Big\}^\alpha dx,~~\alpha>0.
  \end{align}
 For the estimation of (\ref{eq5.4}) and (\ref{eq5}), the unknown model parameters $k$ and $\lambda$ are estimated using the  maximum likelihood method. The maximum likelihood estimators (MLEs) of RIGF in (\ref{eq5.4}) and IGF in (\ref{eq5}) are then obtained as
\begin{align}\label{eq5.5}
\widehat R^\alpha_\beta(X)=\delta(\alpha)\left(\int_{0}^{\infty}\Big\{\frac{\widehat k}{\widehat\lambda}\Big(\frac{x}{\widehat\lambda}\Big)^{\widehat k-1}e^{-(\frac{x}{\widehat\lambda})^{\widehat k}}\Big\}^\alpha dx\right)^{\beta-1},~~0<\alpha<\infty,~ \alpha\neq1,~ \beta>0,
\end{align} 
and
\begin{align}\label{eq5*}
\widehat G_\alpha(X)=\int_{0}^{\infty}\Big\{\frac{\widehat k}{\widehat\lambda}\Big(\frac{x}{\widehat\lambda}\Big)^{\widehat k-1}e^{-(\frac{x}{\widehat\lambda})^{\widehat k}}\Big\}^\alpha dx,~~\alpha>0,
\end{align} 
where $\widehat k$ and $\widehat \lambda$ are the MLEs of the unknown model parameters $k$ and $\lambda$, respectively. To obtain the SD, AB, and MSE values of $\widehat R^\alpha_\beta(X)$ in (\ref{eq5.5}) and $\widehat G_\alpha(X)$ in (\ref{eq5*}), we carry out a Monte Carlo simulation using R software with $500$ replications. The SD, AB, and MSE values are then obtained for  different choices of parameters $\alpha$ (for fixed $\beta=1.1$, $k=2$ and $\lambda=1.5$), $\beta$ (for fixed $\alpha=0.3$, $k=2$ and $\lambda=1$), $k$ (for fixed $\alpha=0.3$, $\beta=0.5$ and $\lambda=1$), $\lambda$ (for fixed $\alpha=0.3$, $\beta=0.5$ and $k=2$), and sample sizes $n=150, 300, 500$. We have presented the SD, AB, and MSE in Tables \ref{tb8} and \ref{tb9}. We observe the following:
\begin{itemize}
\item The values of the SD, AB, and MSE decrease as sample size $n$ increases for all cases of the parameters $\alpha, \beta$, $k$ and $\lambda$;
\item In general, the SD, AB, and MSE values of the parametric estimator of the RIGF are lesser than those of the IGF, implying a better performance of the estimator of the proposed RIGF than IGF;
\item Similar behaviour is observed for other choices of the parameters;
\item It is observed from Tables \ref{tb4}-\ref{tb9}  that the parametric estimator in (\ref{eq5.5}) performs better than the non-parametric estimator in (\ref{eq5.2}) based on the values of AB and MSE  for Weibull distribution, as one would expect.
\end{itemize}

\begin{table}[ht!]
\caption {Comparison  between  the parametric estimators of the IGF in (\ref{eq5*}) and RIGF in (\ref{eq5.5}) in terms of the SD, AB and MSE for different choices of $\alpha,~\beta,~k,~\lambda$ and $n$.}
	\begin{center}

	\scalebox{0.7}{\begin{tabular}{c c c c c c c c c c c c c c } 
			\hline\hline 
		\textbf{$\alpha$} & $\textbf{n}$ & \multicolumn{4}{c}{\bf{$\beta=1.1,~\lambda=1.5,~k=2$}}& \textbf{$\beta$}& $\textbf{n}$ & \multicolumn{4}{c}{$\alpha=0.3,~\lambda=1,~k=2$ }  \\
			\hline
			&&	\multicolumn{2}{c}{\textbf{IGF}}& \multicolumn{2}{c}{\textbf{RIGF}} &&&\multicolumn{2}{c}{\textbf{IGF}}& \multicolumn{2}{c}{\textbf{RIGF} } \\
						\hline
			&& 	 \textbf{SD}& \textbf{AB}& \textbf{SD} & \textbf{AB} &&& \textbf{SD} & \textbf{AB} & \textbf{SD} &\textbf{ AB}  \\
			 &~&~&\textbf{(MSE)} & ~& \textbf{(MSE)}&~&~ & ~ &\textbf{(MSE)} &~ & \textbf{(MSE)} \\
			\hline\hline
\multirow{10}{1.9cm}
~ & $150$ & $ 0.10209$ & $0.01055$& $0.00649$ &$0.00079$&~& 150 & $ 0.07686$& $0.00794$& $0.02464$& $0.00333$  \\[0.5ex]
~& ~& ~ & $( 0.01053)$&~  & $(0.00004)$ &~&~&~  & $(  0.00597)$ &~  & $(0.00062)$   \\[1.2ex]
$0.3$ & $300$ & $0.07075$ & $0.00391$& $0.00448$ & $0.00030$ & 0.4 & 300& $0.05327$& $0.00294$& $0.01697$& $0.00132$   \\[0.5ex]
~& ~& ~ & $(0.00502)$&~  & $(0.00002)$ &~ &~&~ & $(0.00285)$ &~  & $(0.00029)$  \\[1.2ex]
~ & $500$ & $0.05482$ & $0.00177$& $0.00347$ &$0.00015$&~ & 500& $0.04127$& $0.00133$& $0.01313$& $0.00065$   \\[0.5ex]
~& ~& ~ & $(0.00301)$&~  & $(0.00001)$&~&~ &~  & $(0.00171)$ &~  & $( 0.00017)$  \\[1.2ex]

\hline

\multirow{10}{1.9cm}
~ & $150$ & $ 0.01250$ & $0.00176$& $0.00519$ &$0.00075$&~& 150 & $0.07686$& $0.00794$& $0.01857$& $0.00244$   \\[0.5ex]
~& ~& ~ & $(0.00016)$&~  & $(0.00003)$ &~&~&~  & $(0.00597)$ &~  & $(0.00035)$  \\[1.2ex]
$0.8$ & $300$ & $0.00867$ & $0.00068$& $0.00359$ & $0.00029$ & 0.6 & 300& $0.05327$& $0.00294$& $0.01280$& $0.00096$   \\[0.5ex]
~& ~& ~ & $(0.00008)$&~  & $(0.000013)$  &~ &~&~ & $(0.00285)$ &~  & $( 0.00016)$  \\[1.2ex]
~ & $500$ & $0.00672$ & $0.00035$& $0.00279$ &$0.00015$&~ & 500& $0.04127$& $0.00133$& $0.00990$& $0.00047$   \\[0.5ex]
~& ~& ~ & $( 0.00005)$&~  & $(0.000008)$&~&~ &~  & $(0.00171)$ &~  & $(0.00010)$  \\[1.2ex]

\hline

\multirow{10}{1.9cm}
~ & $150$ & $0.00815$ & $0.00129$& $0.00484$ &$0.00075$&~& 150 & $0.07686$& $0.00794$& $0.00558$& $0.00070$  \\[0.5ex]
~& ~& ~ & $(0.00007)$&~  & $(0.00002)$ &~&~&~  & $(0.00597)$ &~  & $(0.00003)$   \\[1.2ex]
$1.2$ & $300$ & $0.00564$ & $0.00050$& $0.00335$ & $0.00029$  & 0.9 & 300& $0.05327$& $0.00294$& $0.00385$& $0.00027$   \\[0.5ex]
~& ~& ~ & $(0.00003)$&~  & $(0.000011)$ &~ &~&~ & $( 0.00285)$ &~  & $(0.00001)$  \\[1.2ex]
~ & $500$ & $0.00437$ & $0.00027$& $0.00260$ &$0.00015$&~ & 500& $ 0.04127$& $0.00133$& $ 0.00298$& $0.00013$   \\[0.5ex]
~& ~& ~ & $( 0.00002)$&~  & $(0.000007)$&~&~ &~  & $(0.00171)$ &~  & $(0.000009)$  \\[1.2ex]

\hline

\multirow{10}{1.9cm}
~ & $150$ & $ 0.01549$ & $0.00261$& $0.00466$ &$0.00074$&~& 150 & $ 0.07686$& $0.00794$& $0.01342$& $0.00160$  \\[0.5ex]
~& ~& ~ & $(0.00025)$&~  & $(0.00002)$ &~&~&~  & $( 0.00597)$ &~  & $(0.00018)$   \\[1.2ex]
$1.5$ & $300$ & $0.01070$ & $0.00103$& $0.00323$ & $0.00029$& 1.2 & 300& $0.05327$& $0.00294$& $0.00927$& $0.00062$   \\[0.5ex]
~& ~& ~ & $(0.00012)$&~  & $(0.00001)$  &~ &~&~ & $(0.00285)$ &~  & $(0.00009)$  \\[1.2ex]
~ & $500$ & $0.00828$ & $0.00056$& $0.00250$ &$0.00016$&~ &  500& $0.04127$& $0.00133$& $0.00718$& $0.00030$   \\[0.5ex]
~& ~& ~ & $( 0.00007)$&~  & $(0.000006)$&~&~ &~  & $( 0.00171)$ &~  & $(0.00005)$  \\[1.2ex]

\hline

\multirow{10}{1.9cm}
~ & $150$ & $ 0.02053$ & $0.00373$& $0.00443$ &$0.00073$&~& 150 & $0.07686$& $0.00794$& $0.05148$& $0.00574$  \\[0.5ex]
~& ~& ~ & $(0.00044)$&~  & $(0.00002)$ &~&~&~  & $(0.00597)$ &~  & $( 0.00268)$   \\[1.2ex]
$2.0$ & $300$ & $0.01410$ & $0.00151$& $0.00307$ & $0.00028$ & 1.6 & 300& $0.05327$& $0.00294$& $0.03564$& $0.00217$   \\[0.5ex]
~& ~& ~ & $(0.00020)$&~  & $(0.00001)$ &~ &~&~ & $(0.00285)$ &~  & $(0.00127)$  \\[1.2ex]
~ & $500$ & $0.01091$ & $0.00083$& $0.00238$ &$0.00015$&~ & 500& $0.04127$& $0.00133$& $0.02761$& $0.00101$   \\[0.5ex]
~& ~& ~ & $( 0.00012)$&~  & $(0.00001)$&~&~ &~  & $(0.00171)$ &~  & $( 0.00076)$  \\[1.2ex]

\hline \hline
		\label{tb8} 		
	\end{tabular}} 
	\end{center}
	\end{table}

\begin{table}[ht!]
\caption {Continuation of Table \ref{tb8}.}
	\begin{center}

	\scalebox{0.7}{\begin{tabular}{c c c c c c c c c c c c c c } 
			\hline\hline 
		\textbf{$k$} & $\textbf{n}$ & \multicolumn{4}{c}{\bf{$\alpha=0.3,~\beta=0.5,~\lambda=1$}}& \textbf{$\lambda$}& $\textbf{n}$ & \multicolumn{4}{c}{$\alpha=0.3,~\beta=0.5,~k=2$ }  \\
			\hline
			&&	\multicolumn{2}{c}{\textbf{IGF}}& \multicolumn{2}{c}{\textbf{RIGF}} &&&\multicolumn{2}{c}{\textbf{IGF}}& \multicolumn{2}{c}{\textbf{RIGF} } \\
						\hline
			&& 	 \textbf{SD}& \textbf{AB}& \textbf{SD} & \textbf{AB} &&& \textbf{SD} & \textbf{AB} & \textbf{SD} &\textbf{ AB}  \\
			 &~&~&\textbf{(MSE)} & ~& \textbf{(MSE)}&~&~ & ~ &\textbf{(MSE)} &~ & \textbf{(MSE)} \\
			\hline\hline
\multirow{10}{1.9cm}
~ & $150$ & $  1.95725$ & $0.02242$& $0.03669$ &$0.00505$&~& 150 & $ 0.0714$& $0.00738$& $0.02265$& $0.00302$  \\[0.5ex]
~& ~& ~ & $( 3.83132)$&~  & $( 0.00137)$ &~&~&~  & $( 0.00515)$ &~  & $(0.00052)$   \\[1.2ex]
$0.5$ & $300$ & $  1.34515$ & $ 0.01001$& $ 0.02511$ & $ 0.00201$ & 0.9 & 300& $ 0.04948$& $0.00273$& $ 0.01561$& $0.00119$   \\[0.5ex]
~& ~& ~ & $(1.80954)$&~  & $(0.00063)$ &~ &~&~ & $(0.00246)$ &~  & $(0.00024)$  \\[1.2ex]
~ & $500$ & $1.04048$ & $ 0.01776$& $0.01947$ &$0.00099$&~ & 500& $ 0.03834$& $0.00124$& $ 0.01207$& $ 0.00059$   \\[0.5ex]
~& ~& ~ & $(1.08292)$&~  & $( 0.00038)$&~&~ &~  & $( 0.00147)$ &~  & $( 0.00015)$  \\[1.2ex]

\hline

\multirow{10}{1.9cm}
~ & $150$ & $0.42794$ & $0.02286$& $0.03308$ &$0.00413$&~& 150 & $ 0.08733$& $0.00902$& $0.02048$& $ 0.00273$   \\[0.5ex]
~& ~& ~ & $( 0.18365)$&~  & $(0.00111)$ &~&~&~  & $( 0.00771)$ &~  & $( 0.00043)$  \\[1.2ex]
$0.8$ & $300$ & $0.29662$ & $0.00569$& $0.02274$ & $0.00157$ & 1.2 & 300& $0.06052$& $0.00334$& $0.01411$& $0.00108$   \\[0.5ex]
~& ~& ~ & $( 0.08802)$&~  & $(0.00052)$ &~ &~&~ & $(0.00367)$ &~  & $( 0.00040)$  \\[1.2ex]
~ & $500$ & $0.23024$ & $0.00111$& $0.01765$ &$ 0.00077$&~ & 500& $0.04689$& $0.00152$& $0.01092$& $0.00053$   \\[0.5ex]
~& ~& ~ & $(0.05301)$&~  & $( 0.00031)$&~&~ &~  & $(0.00220)$ &~  & $( 0.00012)$  \\[1.2ex]

\hline

\multirow{10}{1.9cm}
~ & $150$ & $0.25018$ & $0.01694$& $0.02981$ &$0.00368$&~& 150 & $0.10681$& $0.01103$& $ 0.01852$& $0.00247$  \\[0.5ex]
~& ~& ~ & $(0.06288)$&~  & $( 0.00090)$ &~&~&~  & $( 0.01153)$ &~  & $(0.00035)$   \\[1.2ex]
$1.0$ & $300$ & $0.17362$ & $0.00494$& $0.02052$ & $0.00138$ & 1.6 & 300& $ 0.07402$& $0.00409$& $ 0.01276$& $ 0.00097$   \\[0.5ex]
~& ~& ~ & $(0.03017)$&~  & $( 0.00042)$ &~ &~&~ & $(  0.00550)$ &~  & $(0.00016)$  \\[1.2ex]
~ & $500$ & $ 0.13486$ & $0.00166$& $0.01593$ &$0.00068$&~ & 500& $0.05735$& $0.00185$& $0.00987$& $ 0.00048$   \\[0.5ex]
~& ~& ~ & $( 0.01819)$&~  & $(0.00025)$&~&~ &~  & $( 0.00329)$ &~  & $(0.00010)$  \\[1.2ex]

\hline

\multirow{10}{1.9cm}
~ & $150$ & $ 0.11604$ & $0.01347$& $0.02437$ &$0.00310$&~& 150 & $ 0.12486$& $0.01290$& $0.01713$& $0.00228$  \\[0.5ex]
~& ~& ~ & $( 0.01357)$&~  & $( 0.00060)$ &~&~&~  & $( 0.01576)$ &~  & $(0.00030)$   \\[1.2ex]
$1.5$ & $300$ & $ 0.08052$ & $0.00356$& $0.01679$ & $0.00118$ & 2.0 & 300& $ 0.08653$& $0.00478$& $ 0.01180$& $0.00090$   \\[0.5ex]
~& ~& ~ & $(0.00650)$&~  & $(0.00028)$ &~ &~&~ & $(0.00751)$ &~  & $( 0.00014)$  \\[1.2ex]
~ & $500$ & $0.06250$ & $0.00153$& $0.01302$ &$0.00058$&~ & 500& $0.06705$& $0.00217$& $ 0.00913$& $0.00045$   \\[0.5ex]
~& ~& ~ & $( 0.00391)$&~  & $(0.00017)$&~&~ &~  & $( 0.0045)$ &~  & $(0.00008)$  \\[1.2ex]

\hline

\multirow{10}{1.9cm}
~ & $150$ & $0.05983$ & $0.00667$& $0.02083$ &$0.00288$&~& 150 & $0.14597$& $0.01508$& $ 0.01584$& $0.00211$  \\[0.5ex]
~& ~& ~ & $(  0.00362)$&~  & $( 0.00044)$ &~&~&~  & $(0.02154)$ &~  & $(0.00026)$   \\[1.2ex]
$2.5$ & $300$ & $ 0.04142$ & $0.00261$& $0.01436$ & $0.00118$ & 2.5 & 300& $0.10116$& $0.00559$& $0.01091$& $0.00083$   \\[0.5ex]
~& ~& ~ & $(0.00172)$&~  & $( 0.00021)$ &~ &~&~ & $( 0.01026)$ &~  & $(0.00012)$  \\[1.2ex]
~ & $500$ & $0.03205$ & $0.00120$& $0.01108$ &$0.00058$&~ & 500& $0.07838$& $0.00254$& $0.00844$& $0.00041$   \\[0.5ex]
~& ~& ~ & $( 0.00103)$&~  & $( 0.00012)$&~&~ &~  & $(0.00615)$ &~  & $( 0.00007)$  \\[1.2ex]

\hline \hline
		\label{tb9} 		
	\end{tabular}} 
	\end{center}
	\end{table}

\section{Real data analysis}\label{sec6}
We consider a real data set related to the failure times (in minutes) of $15$ electronic components in an accelerated life-test. The data set is taken from \cite{lawless2011statistical}, which  is provided in Table \ref{tb5}. For the purpose of numerical illustration, we use here the Gaussian kernel given in (\ref{eq5.3}). Here, we consider four statistical models: exponential (EXP), Weibull, inverse exponential half logistic (IEHL), and log logistic (LL) distributions to check the best fitted model for this data set. The negative log-likelihood criterion $(-\ln L)$, Akaike-information criterion (AIC), AICc, and Bayesian information criterion (BIC) have all been used as measures of fit. From Table \ref{tb6}, we notice that the exponential distribution fits the data set better than other considered distributions since the values of all the measures are smaller than these for other distributions, namely, Weibull, IEHL, and LL. The value of the maximum likelihood estimator (MLE) of the unknown model parameter  $\lambda$ is $0.036279$. We have used $500$ bootstrap samples with size $n=15$ and choose $\beta_n=0.35$ for computing purpose. The values of AB and MSE for different choices of $\alpha$ (for fixed $\beta=2.5$) and $\beta$ (for fixed $\alpha=3.5$) are presented in Table \ref{tb7}. We observe that the values of AB and MSE all become smaller for  larger values of $n$, verifying the consistency  of the proposed estimator.
\begin{table}[ht!]
	\caption {The data set on failure times (in minutes), of electronic components.}
	\centering 
	\scalebox{0.85}{\begin{tabular}{c c c c c c c c } 
			\toprule
			1.4,~~5.1,~~6.3,~~10.8,~~12.1,~~18.5,~~19.7,~~22.2,~~23.0,~~30.6,~~37.3,~~46.3,~~53.9,~~59.8,~~66.2. \\
			\bottomrule
			\label{tb5} 			
	\end{tabular}} 
\end{table}

\begin{table}[ht!]
	\centering
	\caption {{The MLEs, BIC, AICc, AIC, and negative log-likelihood values of some statistical models for the real data set in Table \ref{tb5}}.}
	\scalebox{0.85}{\begin{tabular}{ccccccc}
			\toprule
			\textbf{Model}  & \textbf{Shape}  & \textbf{Scale}  & \textbf{-ln L}  & \textbf{AIC}& \textbf{AICc} & \textbf{BIC}  \\
			\midrule
			EXP  & $\widehat{\lambda}= 0.036279$ & ~  &  64.7382 &  131.4765 &  131.7841 &  132.1845\\[1.2ex]
			Weibull  & $\widehat{\alpha}= 1.008962$ & $\widehat{\lambda}=50.68767$  & 67.01285 &  138.0257 &  139.0257&  139.4418\\[1.2ex]
			IEHL  & $\widehat{\alpha}= 0.69014$ & $\widehat{\lambda}=0.0099735$  & 70.4478 &  144.8957 &  145.8957&  146.3118\\[1.2ex]
			LL &  $\widehat{\alpha}=1.751468$ & $\widehat{\lambda}=20.82626$ & 173.1330 & 350.2659&   351.1659 &  351.6820 \\[1 ex]
			\bottomrule
			\label{tb6}
	\end{tabular}}
\end{table}

\begin{table}[ht!]
	\centering 
	\caption{The AB, MSE of the non-parametric estimator of the RIGF and the value of $R^\alpha_\beta(X)$ based on the real data set in Table \ref{tb5} for different choices of $\alpha$ (for fixed $\beta=2.5$) and $\beta$ (for fixed $\alpha=3.5$).}
	\scalebox{0.85}{\begin{tabular}{c c c| c c c} 
			\hline\hline\vspace{.1cm} 
			$\alpha$&  \textbf{AB} & $R^\alpha_\beta(X)$ &$\beta$  &\textbf{AB}
			&$R^\alpha_\beta(X)$ \\ [0.5ex] 
			$(\beta=2.5)$ &\textbf{(MSE)} & ~	&$(\alpha=3.5)$& \textbf{(MSE)}&~\\
			\hline\hline 
			1.5 & 0.17419   & -0.09050   &1.2& 0.07171        & -0.0593    \\
			~ & (0.03660) &~          & ~ & (0.00552)  &  ~   \\[1ex]
			1.6 & 0.17419   & -0.04163   &1.3& 0.05191       & -0.02283   \\
			~ & (0.01491) &~          & ~ & (0.00301)  &  ~   \\[1ex]
			1.7 & 0.07552   & -0.01981   &1.4& 0.03559        & -0.00879    \\
			~ & (0.00670) &~          & ~ & (0.00144)  &  ~   \\[1ex]
			1.8 & 0.05011   & -0.00967   &1.5&0.02193     & -0.00339   \\
			~ & (0.00303) &~          & ~ & (0.00056)  &  ~   \\[1ex]
			1.9 & 0.03382   & -0.00482   &1.6& 0.01402        & -0.00130    \\
			~ & (0.00140) &~          & ~ & (0.00024)  &  ~   \\[1ex]
			2.0 & 0.02251   & -0.00244   &2.0&0.00192        & -0.00003    \\
			~ & (0.00062) &~          & ~ & (0.00001)  &  ~   \\[1ex]
			2.5 & 0.00359   & -0.00010   &2.5& 0.0001601        & -0.0000002    \\
			~ & (0.00002) &~          & ~ & (0.000000007)  &  ~   \\[1ex]
			3.0 & 0.000667  & -0.000005   &3.0& 0.0000146        & -0.000000002   \\
			~ & (0.000001) &~          & ~ & (0.000000001)  &  ~   \\[1ex]
			\hline\hline 
	\end{tabular}}
	\label{tb7} 
\end{table}

\section{Applications}\label{sec7}
In this section, we discuss some applications of the proposed RIGF. At the end of this section, we highlight that the newly proposed RIGF can be used as an alternative tool to measure uncertainty. First, we discuss its application in reliability engineering.

\subsection*{I.  Application in reliability engineering }
 Coherent systems are essential in both theoretical and practical contexts because they provide a clear and structured way to analyse, design, and model complex systems. Their predictability, robustness, and applicability across various fields make them indispensable in ensuring the reliability, safety, and efficiency of systems in real-world applications. Here, we propose the RIGF of coherent systems and discuss its  properties.
 
 We consider a coherent system with $n$ components and lifetime of the coherent system is denoted by $T$. For details of coherent system, one may refer to \cite{navarro2021introduction}. The random lifetimes of $n$ components of the coherent
 system are  identically distributed (i.d.) with a common CDF and PDF $F(\cdot)$  and $f(\cdot)$, respectively. The CDF and PDF of $T$ are defined as
 \begin{align}
 F_T(x)=q(F(x)) ~~~~~\text{and}~~~~~f_T(x)=q'(F(x))f(x),
 \end{align}
respectively, where $q:[0,1]\rightarrow[0,1]$ is the distortion function (see \cite{navarro2013stochastic}) and $q'\equiv\frac{dq}{dx}$. We recall that the distortion function depends on the structure of a system and the copula of the component lifetimes. It is increasing and continuous function with $q(0)=0$ and $q(1)=1$. Several researchers studied various information measures for coherent systems. In this direction, readers may refer to \cite{toomaj2017some}, \cite{cali2020properties},  \cite{saha2023extended}, and \cite{saha2024weighted}. The RIGF of $T$ can be expressed as
\begin{align}\label{eq5.2*}
R_\beta^\alpha(T)
=\delta(\alpha)\bigg(\int_{0}^{\infty}\psi_\alpha\big(F_T(x)\big)dx\bigg)^{\beta-1}
=\delta(\alpha)\bigg(\int_{0}^{1}\frac{\psi_\alpha\big(q(u)\big)}{f(F^{-1}(u))}du\bigg)^{\beta-1},
\end{align}
where $\psi_\alpha(u)=f^\alpha_T(F^{-1}_T(u))$, for $0\le u\le1$. 

Next, we consider an example to obtain the RIGF of a coherent system.  
\begin{example}\label{ex6.1}
Suppose $X_1,X_2,$ and $X_3$ denote the independent lifetimes of the components of a coherent system. Assume that they all follow power distribution with CDF $F(x)=x^a,~x\in[0,1]$ and $a>0$. We take a parallel system with lifetime $T=X_{3:3}=\max\{X_1,X_2,X_3\}$ whose distortion function is $q(u)=u^3,~0\le u\le1$. Thus, from (\ref{eq5.2*}), the RIGF of the coherent system,  for $0<\alpha<\infty,~\alpha \neq1$ and $\beta>0$, is obtained as
\begin{align*}
R^\alpha_\beta(T)=\delta(\alpha)\Big\{\frac{(3a)^\alpha}{1+2\alpha a(a-1)}\Big\}^{\beta-1}.
\end{align*}
\end{example}

In order to check the behaviour of the RIGF of a coherent system with respect to $\beta$ in Example \ref{ex6.1}, its graphs are plotted in Figure \ref{fig4} for different values of $a$. 

\begin{figure}[htbp!]
  	\centering
  	\subfigure[]{\label{c1}\includegraphics[height=2in]{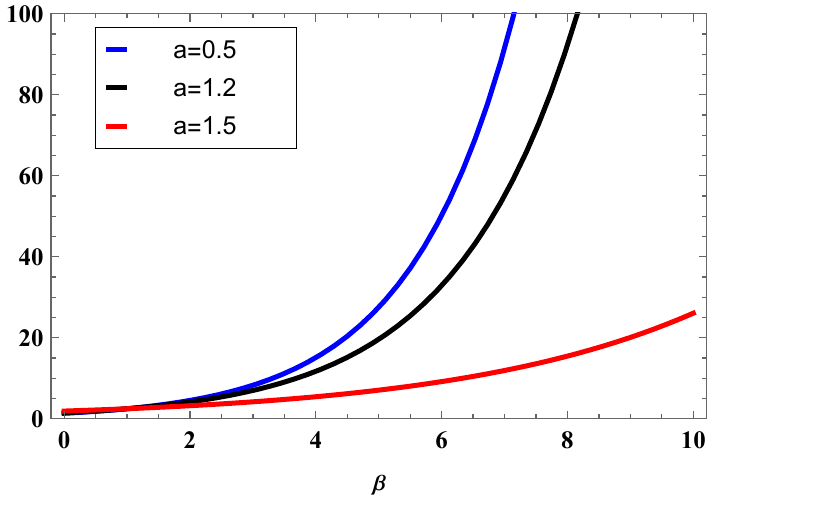}}
  	\subfigure[]{\label{c1}\includegraphics[height=2in]{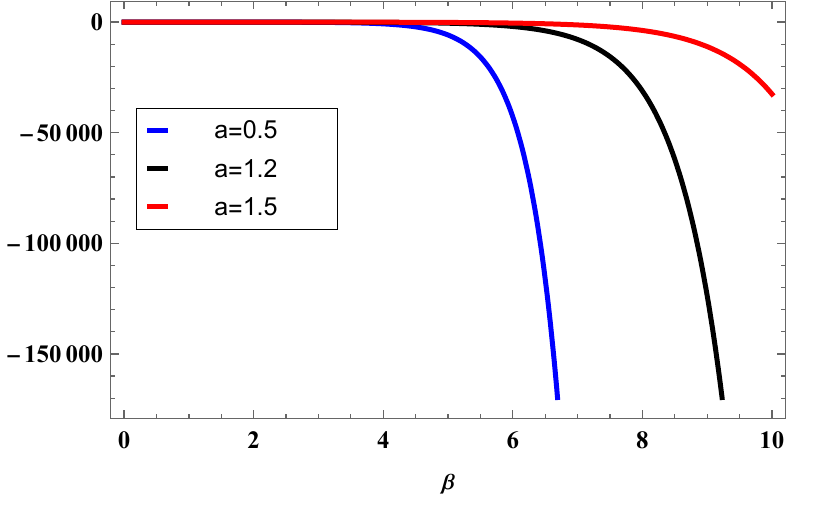}}
  	\caption{Graphs of the RIGF of parallel system for $(a)$ $\alpha=0.6$ and $(b)$ $\alpha=1.5$  in Example \ref{ex6.1}. Here, we have considered $a=0.5,~1.2,~1.5$.}
  \label{fig4}
  \end{figure}

Next, we establish relationship between the RIGF of a coherent system and that of its components.
\begin{proposition}
Suppose $T$ is the lifetime of a coherent system with identically distributed components and $q(\cdot)$ is a distortion function. Assume that $X$ is the component lifetime of the coherent system with CDF and PDF $F(\cdot)$ and $f(\cdot)$, respectively, and $\psi_\alpha(u)=f^\alpha_T(F^{-1}_T(u))$, $\phi_\alpha(u)=f^\alpha(F^{-1}(u))$. If $\psi_\alpha(q(u))\ge(\le)\phi_\alpha(u)$ for $0\le u\le 1$, then
\begin{equation*}
		R^\alpha_\beta(T)\left\{
		\begin{array}{ll}
			 \ge (\le) R^\alpha_\beta(X),~	
			\text{for}~~\{\alpha>1,\beta\le1\}~\text{or}~\{0<\alpha<1,\beta\ge1\},
			\\
			\le (\ge) R^\alpha_\beta(X),~\text{for}~~\{\alpha>1,\beta\ge1\}~\text{or}~\{0<\alpha<1,\beta\le1\}.
	 \end{array}
		\right.
	\end{equation*}
\end{proposition}

\begin{proof}
Consider $0<\alpha<1$, $\beta\ge1$ and $\psi_\alpha(q(u))\ge\phi_\alpha(u)$. Then,
\begin{align}\label{eq5.4*}
\frac{\psi_\alpha(q(u))}{f(F^{-1}(u))}\ge \frac{\phi_\alpha(u)}{f(F^{-1}(u))}
\Rightarrow\delta(\alpha)\bigg(\int_{0}^{1}\frac{\psi_\alpha(q(u))}{f(F^{-1}(u))}du\bigg)^{\beta-1}&\ge\delta(\alpha)\bigg(\int_{0}^{1} \frac{\phi_\alpha(u)}{f(F^{-1}(u))}du\bigg)^{\beta-1},
\end{align}
from which the result $R^\alpha_\beta(T)\ge R^\alpha_\beta(X)$ follows directly.  Proofs  for other cases are similar and are therefore omitted. 
\end{proof}

In the following proposition, we establish that two coherent systems are comparable based on the proposed generating function. The dispersive ordering has been used for this purpose.
\begin{proposition}
Let $T_1$ and $T_2$ be the lifetimes of two different coherent systems with same structure and respective identically distributed component lifetimes $X_1,\cdots,X_n$ and $Y_1,\cdots,Y_n$ with the same copula. The common CDFs and PDFs for $X_1,\cdots,X_n$ and $Y_1,\cdots,Y_n$ are $F_X(\cdot)$, $f_X(\cdot)$ and $F_Y(\cdot)$, $f_Y(\cdot)$, respectively, and $\psi_\alpha(u)=f^\alpha_T(F^{-1}_T(u)),~0\le u\le1$. If $X\le_{disp}Y$, then

\begin{itemize}

\item [$(A)$] $R^\alpha_\beta(T_1)\le R^\alpha_\beta(T_2),$  for $\{\alpha>1,\beta\le1\}$ or $\{0<\alpha<1,\beta\ge1\}$,
\item[$(B)$] $R^\alpha_\beta(T_1)\ge R^\alpha_\beta(T_2),$ for $\{\alpha>1,\beta\ge1\}$ or $\{0<\alpha<1,\beta\le1\}$.
\end{itemize}
\end{proposition}

\begin{proof}
$(A)$ Note that both systems with lifetimes $T_1$ and $T_2$ have a common distortion function $q(\cdot)$, since the systems have the same structure and the same copula. Under the assumption, we have $X\le_{disp}Y$, which implies  $f_X(F_X^{-1}(u))\ge f_Y(F_Y^{-1}(u)),~\forall~ 0\le u\le 1$. Thus,
\begin{align}\label{eq7.5}
\frac{\psi_\alpha(q(u))}{f_X(F^{-1}_X(u))}\le \frac{\psi_\alpha(q(u))}{f_Y(F^{-1}_Y(u))}.
\end{align}
Hence, the result follows directly from (\ref{eq7.5}). Hence the required result.
\\
$(B)$ The proof is quite similar to that of Part $(A)$, and is therefore not presented here. 
\end{proof}

Next, we obtain bounds of the RIGF $R^\alpha_\beta(T)$ in terms of  $R^\alpha_\beta(X)$ when a coherent system has identically distributed components. 
\begin{proposition}\label{prop5.3}
Suppose that $T$ and $X$ are respectively the lifetimes of a coherent system and the component of this coherent system. Further, assume that the coherent system has identically distributed components with CDF $F(\cdot)$ and PDF $f(\cdot)$,and its distortion function is $q(\cdot)$. Take $\psi_\alpha(u)=f^\alpha_T(F^{-1}_T(u))$ and $\phi_\alpha(u)=f^\alpha(F^{-1}(u)),$ for $0\le u\le1$. Then, we have 
\begin{itemize}
\item[$(A)$]$\xi_{1,\alpha}R^\alpha_\beta(X)\le R^\alpha_\beta(T)\le \xi_{2,\alpha}R^\alpha_\beta(X)$, for $\{\alpha>1,\beta\le1\}$ or $\{0<\alpha<1,\beta\ge1\}$,
\item[$(B)$] $\xi_{1,\alpha}R^\alpha_\beta(X)\ge R^\alpha_\beta(T)\ge \xi_{2,\alpha}R^\alpha_\beta(X)$, for $\{\alpha>1,\beta\ge1\}$ or $\{0<\alpha<1,\beta\le1\}$.
\end{itemize}
where $\xi_{1,\alpha}=\big(\inf_{u\in(0,1)}\frac{\psi_\alpha(q(u))}{\phi_\alpha(u)}\big)^{\beta-1}$ and $\xi_{2,\alpha}=\big(\sup_{u\in(0,1)}\frac{\psi_\alpha(q(u))}{\phi_\alpha(u)}\big)^{\beta-1}$.
\end{proposition}
\begin{proof}
$(A)$~From (\ref{eq5.2*}), we obtain
\begin{align*}
R^\alpha_\beta(T)&=\delta(\alpha)\bigg(\int_{0}^{1}\frac{\psi_\alpha\big(q(u)\big)}{f(F^{-1}(u))}du\bigg)^{\beta-1}\\
&=\delta(\alpha)\bigg(\int_{0}^{1}\frac{\psi_\alpha(q(u))}{\phi_\alpha(u)}\times\frac{\phi_\alpha(u)}{f(F^{-1}(u))}du\bigg)^{\beta-1}\\
&\le\bigg(\sup_{u\in(0,1)}\frac{\psi_\alpha(q(u))}{\phi_\alpha(u)}\bigg)^{\beta-1}\times\delta(\alpha)\bigg(\int_{0}^{1}\frac{\phi_\alpha(u)}{f(F^{-1}(u))}du\bigg)^{\beta-1}=\xi_{2,\alpha}R^\alpha_\beta(X).
\end{align*}
Hence, the proof of the right side inequality is completed. The proof of the left side inequality is similar, and is therefore omitted.
\\
$(B)$ The proof is quite similar to that of Part $(A)$, and is  therefore omitted.
\end{proof}

The following proposition shows that the preceeding result can be extended to compare two systems based on the RIGF.

\begin{proposition}
Let $T_1$ and $T_2$ be the lifetimes of two coherent systems with identically distributed components with distortion functions $q_1$ and $q_2$, respectively. Assume that   $\psi_\alpha(u)=f^\alpha_T(F^{-1}_T(u)),$ for $0\le u\le 1$.  Then,
\begin{itemize}
\item[$(A)$] $\gamma_{1,\alpha}R^\alpha_\beta(T_1)\le R^\alpha_\beta(T_2)\le \gamma_{2,\alpha}R^\alpha_\beta(T_1)$, for $\{\alpha>1,\beta\le1\}$ or $\{0<\alpha<1,\beta\ge1\}$,
\item[$(B)$] $\gamma_{1,\alpha}R^\alpha_\beta(T_1)\ge R^\alpha_\beta(T_2)\ge \gamma_{2,\alpha}R^\alpha_\beta(T_1)$, for $\{\alpha>1,\beta\ge1\}$ or $\{0<\alpha<1,\beta\le1\}$,
\end{itemize}
where $\gamma_{1,\alpha}=\big(\inf_{u\in(0,1)}\frac{\psi_\alpha(q_2(u))}{\psi_\alpha(q_1(u))}\big)^{\beta-1}$ and $\gamma_{2,\alpha}=\big(\sup_{u\in(0,1)}\frac{\psi_\alpha(q_2(u))}{\psi_\alpha(q_1(u))}\big)^{\beta-1}$.
\end{proposition}

\begin{proof}
The proof is similar to that of Proposition \ref{prop5.3}, and is therefore omitted for brevity.
\end{proof}

The following result provides additional bounds of the RIGF of the lifetime of a coherent system when the PDF is bounded.  The proof is simple, and thus it is omitted.

\begin{proposition}
Consider a coherent system as in Proposition \ref{prop5.3}. Let the CDF and PDF of the components be $F(\cdot)$ and $f(\cdot)$, respectively, and $\psi_\alpha(u)=f^\alpha_T(F^{-1}_T(u))$, for $0\le u\le1,~0<\alpha<\infty,~\alpha\neq1$ and $\beta>0$. 
\begin{itemize}
\item[$(A)$] If $f(x)\le M,~ \forall x\in S$, then $R^\alpha_\beta(T)\ge(\le)\frac{1}{M^{\beta-1}}\big(\int_{0}^{1}\psi_\alpha(q(u))du\big)^{\beta-1}$, for $\alpha<(>)1$;
\item[$(B)$] If $f(x)\ge L>0,~ \forall x\in S$, then $R^\alpha_\beta(T)\le (\ge)\frac{1}{L^{\beta-1}}\big(\int_{0}^{1}\psi_\alpha(q(u))du\big)^{\beta-1},$ for $\alpha<(>)1$.
\end{itemize}
\end{proposition}

Next, a comparative study is carried out between the proposed RIGF and IGF (due to \cite{golomb1966}), R\'enyi entropy (due to \cite{renyi1961measures}), varentropy (due to \cite{fradelizi2016optimal}) for three different coherent systems with three components. Suppose  $T$ and $X$ denote the  system's lifetime and component's lifetime  with PDFs $f_T(\cdot)$ and $f(\cdot)$ and CDFs $F_T(\cdot)$ and $F(\cdot)$, respectively. The IGF and R\'enyi entropy of $T$ are 
\begin{align}\label{eq5.7*}
I(T)=\int_{0}^{\infty}f^\alpha_T(x)dx
=\int_{0}^{1}\frac{\psi_\alpha\big(q(u)\big)}{f(F^{-1}(u))} du,~~\alpha>0,
 \end{align}
 and
\begin{align}\label{eq5.8}
H_{\alpha}(T)=\delta(\alpha)\log\int_{0}^{\infty}f_T^{\alpha}(x)dx
=\delta(\alpha)\log\int_{0}^{1}\frac{\psi_\alpha\big(q(u)\big)}{f(F^{-1}(u))}du,~~\alpha>0~(\ne1),
\end{align}
where $\psi_\alpha(q(u))=f^\alpha_T(F^{-1}_T(q(u)))$, respectively.  Further, the varentropy of $T$ is 
\begin{align}\label{eq5.9}
VE(T)&=\int_{0}^{\infty}f_T(x)\Big(\log f_T(x)\Big)^2dx-\Big\{\int_{0}^{\infty}f_T(x)\log f_T(x)dx\Big\}^2\nonumber\\
&=\int_{0}^{\infty}\frac{\psi_1\big(q(u)\big)}{f(F^{-1}(u))}\Big(\log \frac{\psi_1\big(q(u)\big)}{f(F^{-1}(u))}\Big)^2dx-\Big\{\int_{0}^{\infty}\frac{\psi_1\big(q(u)\big)}{f(F^{-1}(u))}\log \frac{\psi_1\big(q(u)\big)}{f(F^{-1}(u))}dx\Big\}^2,
\end{align}
where $\psi_1(q(u))=f_T(F^{-1}_T(q(u))).$ Here, we consider the power distribution with CDF $F(x)=\sqrt{x},~x>0$, as a baseline distribution (component lifetime) for illustrative purpose. We take three coherent systems: series system ($X_{1:3}$), 2-out-of-3 system ($X_{2:3}$), and parallel system ($X_{3:3}$) for evaluating the values of $R^\alpha_\beta(T)$ in (\ref{eq5.2*}), $I(T)$ in (\ref{eq5.7*}), $H_{\alpha}(T)$ in (\ref{eq5.8}), and $VE(T)$ in (\ref{eq5.9}). The numerical values of the RIGF, IGF, R\'enyi entropy, and varentropy for the series, 2-out-of-3, and parallel systems with $\alpha=1.2$ and $\beta=0.5$ are reported in Table \ref{tb4*}. As expected, from Table \ref{tb4*}, we observe that the uncertainty values of the series system are maximum; and minimum for parallel system considering all information measures, validating the proposed information generating function.

\begin{table}[ht!]
	\centering
	\caption {The values of the RIGF, IGF, R\'enyi entropy, and varentropy for the series, 2-out-of-3, and parallel systems.}
	\scalebox{0.95}{\begin{tabular}{ccccccc}
			\toprule
			\textbf{System}  & \textbf{RIGF}  & \textbf{IGF}  & \textbf{R\'enyi entropy}  & \textbf{Varentropy}  \\
			\midrule
			Series ($X_{1:3}$)  & $-4.144032$ & $1.455774$  &  $0.7510748$ &  $2.940702$ \\[1.2ex]
		2-out-of-3 ($X_{2:3}$) & $-4.852534$ & $1.061702$  & $0.1197473$&  $0.194906$ \\[1.2ex]
		   	Parallel ($X_{3:3}$) & $-4.958784$ & $1.016692$  & $0.03310925$ &  $0.1111111$ \\[1.2ex]
		   
			\bottomrule
			\label{tb4*}
	\end{tabular}}
\end{table}

\subsection*{II. RDIGF and RIGF as model selection criteria}
Here, we show that the proposed information generating functions RDIGF and RIGF can  be used as model selection criteria. First, we focus on  RDIGF.  
In this regard, we consider the real data set, dealing with the failure times (in minutes) of electronic components in an accelerated life-test, given in Table \ref{tb5}. We  conduct a goodness of fit test here. The four statistical models: EXP, Weibull, IEHL, and LL distributions are considered for the test. The values of the test statistics: -ln L, AIC, AICc, BIC, and MLEs of the unknown model parameters  are computed, and are given in Table \ref{tb6}. From Table \ref{tb6}, we notice that the exponential distribution fits better than other distributions. The sequence of fitness of the statistical models is EXP, Weibull, IEHL, and LL distributions according the values of -ln L, AIC, AICc, and BIC. Now, we obtain the values of RDIGF between EXP and Weibull (denoted as RDIGF$(E,W)$),  EXP and IEHL (denoted as RDIGF$(E,I)$); and EXP and LL (denoted as RDIGF$(E,L)$). The values of RDIGF are given for different choices of $\alpha$ and $\beta$ in Table \ref{tb10}. The sequence of the values of RDIGF observed is 
$$\mbox{RDIGF}(E,W)<\mbox{RDIGF}(E,I)<\mbox{RDIGF}(E,L),$$ for different choices of $\alpha$ and $\beta$, as expected.

\begin{table}[ht!]
	\centering
	\caption {The values of RDIGF$(E,W)$, RDIGF$(E,I)$, and RDIGF$(E,L)$, for different choices of $\alpha$ and $\beta$.}
	\scalebox{0.85}{\begin{tabular}{ccccc}
			\toprule
			$\alpha$ ~~~~~& $\beta$ ~~~~~& \textbf{RDIGF(E,W)}  & \textbf{RDIGF(E,I)}  & \textbf{RDIGF(E,L)}  \\
			\midrule
			$0.5$ ~~~~~ & $1.5$ ~~~~~&  0.21880 &   0.95542 &  0.97536\\[1.2ex]
		    $0.8$ ~~~~~ & $1.5$ ~~~~~&  0.06195 &   0.25422 &  0.27304\\[1.2ex]
		    $1.5$ ~~~~~ & $1.5$ ~~~~~& 0.00419 &   0.01506 &  0.01684\\[1.2ex]
		    $2.0$ ~~~~~ & $1.5$ ~~~~~&  0.00067 &   0.00225 &  0.00250\\[1.2ex]
		    $2.5$ ~~~~~ & $1.5$ ~~~~~&  0.00011 &   0.00035&  0.00038\\[1.2ex]
		    $0.7$ ~~~~~ & $0.9$ ~~~~~&  0.34505 &   6.20304 &  6.28627\\[1.2ex]
           $0.7$ ~~~~~ & $1.2$ ~~~~~&  0.83185 &    1.52512 &  1.52816\\[1.2ex]
           $0.7$ ~~~~~ & $1.8$ ~~~~~&  0.01202 &   0.10296 &  0.12006\\[1.2ex]
           $0.7$ ~~~~~ & $2.5$ ~~~~~&   0.00016 &   0.00492 &  0.00771\\[1.2ex]
           $0.7$ ~~~~~ & $3.5$ ~~~~~&  0.00001 &   0.00058 &  0.00116\\[1.2ex]   
			
			\bottomrule
			\label{tb10}
	\end{tabular}}
\end{table}

Next, we conduct a Monte Carlo simulation study to demonstrate the importance of the RIGF for the purpose of  model selection. Firstly, using $R$ software, we generate $500$ exponentially distributed random variates with $\lambda=0.5.$ Then, we compute RIGF of this data set under the assumption that the same set  of data comes from 
	\begin{itemize}
		\item exponential distribution;
		\item Weibull distribution;
		\item Pareto distribution. 
	\end{itemize}
For this purpose, the maximum likelihood estimates have been used to calculate the estimated value of RIGF. This process is repeated $1000$ times, and then the favourable proportions in each case have been counted. The results so obtained are presented in Table \ref{tb11}, from which we observe that proportion of the RIGF for exponential model is larger than other proportional values, as expected. From the tabulated values, we also observe that the Weibull distribution will be the better fitting choice than Pareto distribution.

\begin{table}[ht!]
	\centering
	\caption { The proportion of the  values of the RIGF for exponential, Weibull and Pareto distributions.}
	\scalebox{0.85}{\begin{tabular}{ccccc}
			\toprule
			\textbf{$\alpha$} ~~~~~& \textbf{$\beta$} ~~~~~& \textbf{proportion (exponential)}  & \textbf{proportion (Weibull)}  & \textbf{proportion (Pareto)}  \\
			\midrule
			$0.8$ ~~~~~ & $0.7$ ~~~~~&  0.381 &   0.363 & 0.256\\[1.2ex]
		    $1.5$ ~~~~~ & $0.7$ ~~~~~&  0.513 &    0.487 &  0.000\\[1.2ex]
		    $2.0$ ~~~~~ & $0.7$ ~~~~~&  0.515 &    0.485 &  0.000\\[1.2ex]
		    $0.9$ ~~~~~ & $0.9$ ~~~~~&   0.405 &     0.391 &  0.204\\[1.2ex]
		    $1.8$ ~~~~~ & $0.9$ ~~~~~&  0.514 &    0.486 &  0.000\\[1.2ex]
		    $0.7$ ~~~~~ & $1.1$ ~~~~~&  0.497 &    0.494 &  0.009\\[1.2ex]

			\bottomrule
			\label{tb11}
	\end{tabular}}
\end{table}

\subsection*{III. Validation using three chaotic maps}
Chaotic maps are mathematical functions that exhibit chaotic behaviour, meaning that they are highly sensitive to initial conditions and can generate complex, seemingly random patterns over time, even though they are deterministic in nature. There are various chaotic maps such as logistic maps, Chebyshev maps, and H\'ennon maps. These maps are used to model complex, real-world systems that exhibit chaotic behaviour, such as weather systems, financial markets, population dynamics in ecology, and even certain physiological processes in biology. In signal processing, chaotic maps can be used for tasks such as compression, encryption, and secure communication. The unpredictability of chaotic signals can help in masking information and making it harder to intercept or decode. These maps are also used in image processing, machine learning, and control theory. Here, we have studied the chaotic behaviour of the proposed RIGF for Chebyshev, H\'ennon, and logistic maps. Very recently, \cite{kharazmi2024fractional} have studied fractional cumulative residual inaccuracy measure in terms of the Chebyshev and logistic maps. 

\subsection*{$(A)$ Chebyshev map}
The Chebyshev map is defined as 
\begin{align}\label{eq6.10}
x_{r+1}=\cos({s^2 \arccos(x_r)}),~~~~r=1,2,\cdots, n-1,
\end{align}
where $x_r\in[-1,1]$ and $s>0$. We have considered the initial value $x_1=0.1$ and sample size $n=10000$. For $0<s\le1$, we get  $x_r\in[0,1]$ and $x_r\in[-1,1]$ for $s>1$.

The bifurcation diagram of the Chebyshev map in (\ref{eq6.10}) is presented in Figure \ref{fig5}$(a)$. Using the data set with size $n=10000$, we have estimated the proposed measure RIGF in (\ref{eq5.2}). Based on the data extracted from the Chebyshev map, two graphs of the RIGF with respect to $\beta$  are provided in Figure \ref{fig5}$(b)$ for  $s=0.8$ and $2.0$ when $\alpha=0.01$. From Figure \ref{fig5}$(a)$, we notice that the chaos for $s>1$ is greater than that for $s<1$.  From Figure \ref{fig5}$(b)$, we observe that the uncertainty computing via the proposed RIGF for $s=2$ (red line) is greater than that when $s=0.8$ (black line) for all $\beta$, as we would expect; also, they are equal when $\beta=1$.

\begin{figure}[h!]
    	\centering
    	\subfigure[]{\label{c1}\includegraphics[height=2.9in]{C1.pdf}}
    	\subfigure[]{\label{c1}\includegraphics[height=2.9in]{CH7.pdf}}
    	\caption{$(a)$ The bifurcation diagram of the Chebyshev map in (\ref{eq6.10}) and $(b)$ the plots of the RIGF for the Chebyshev map when $s=0.8$ (black) and $s=2.0$ (red) with respect to beta ($B$).}
    	\label{fig5}	
    \end{figure}

\subsection*{$(B)$ H\'enon map}
The H\'enon map is a discrete time dynamical system that exhibits chaotic behaviour. Note that  Michel H\'enon first introduced the map as a simplified version of the Poincar\'e section of the Lorenz model. The H\'enon map is defined as 
\begin{align}\label{eq6.11}
x_{i+1}&=y_i+1-ax^2_i;\nonumber\\
y_{i+1}&=bx_i,
\end{align}  
where $i=1,2,\cdots,n-1$, and $a,~b>0$. For details, one may refer to \cite{henon1976two}. Two factors that determine the map's dependability are $a$ and $b$, which for the conventional H\'enon map have values of $a = 1.4$ and $b = 0.3.$ The H\'enon map is chaotic for the classical values. For the other values, the H\'enon map may be chaotic,  intermittent, or converge to a periodic orbit.

Here, we take the initial values $x_1=0.1$ and $y_1=0.1$, sample size $n=10000$ with $b=0.3$ and $\alpha=0.8$. The bifurcation diagram of the H\'ennon map is presented in Figure \ref{fig6}$(a)$. The plots of the  RIGF in (\ref{eq5.2}) based on the H\'enon map have been drawn for $a=1.0$ (black line), $a=1.2$ (blue line) and $a=1.4$ (red line) with fixed parameter value $b=0.3$ in Figure \ref{fig6}$(b)$. From Figure \ref{fig6}$(a)$, we observe  that the chaos is maximum for $a=1.4$. From Figure \ref{fig6}$(b)$, as expected, we get the plots of the RIGF, matching with the bifurcation diagram in Figure \ref{fig6}$(a)$. We also notice that the chaos for $\beta\in[0,1)$ is larger than that for $\beta>1.$

\begin{figure}[h!]
    	\centering
    	\subfigure[]{\label{c1}\includegraphics[height=2.3in]{H2.pdf}}
    	\subfigure[]{\label{c1}\includegraphics[height=2.5in]{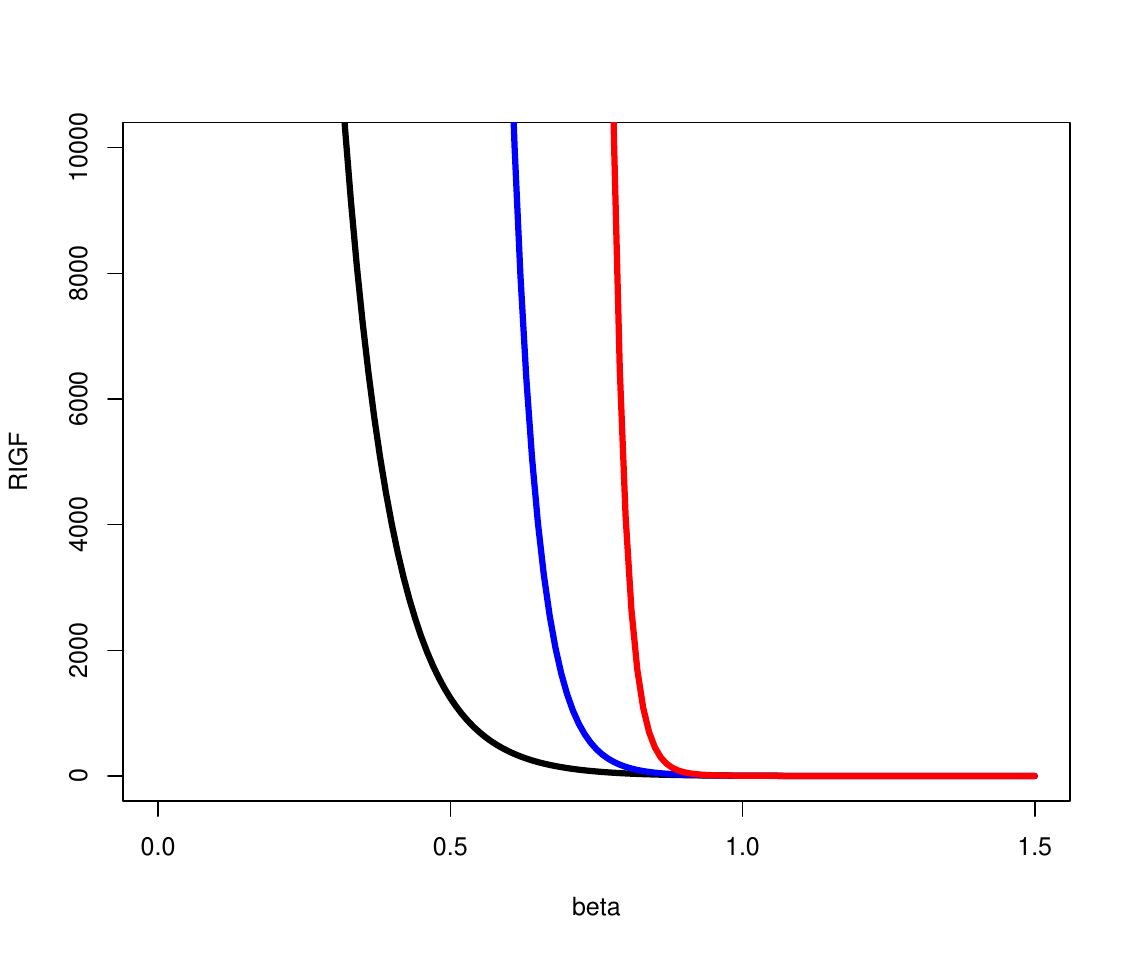}}
    	\caption{$(a)$ Bifurcation diagram of the H\'enon map in (\ref{eq6.11}) and $(b)$ the plots of the RIGF of H\'enon map for $a=1.4$ (red line), $a=1.2$ (blue line), and $a=1.0$ (black line).}
    	\label{fig6}	
    \end{figure}

\subsection*{$(C)$ Logistic map}
The  logistic map used to study the chaotic behaviour of a system is defined by
\begin{align}\label{eq6.12}
x_{i+1}=rx_i(1-x_i),
\end{align}
where $i=1,2,\cdots,n-1$, and $0\le r\le4$. For details about this map, see \cite{feigenbaum1978quantitative}. Here, we consider the initial value $x_1=0.1$ and sample size $n=10000$ for the study of the chaotic behaviour of the proposed measure RIGF in (\ref{eq5.2}). The bifurcation diagram of the logistic map in (\ref{eq6.12}) is shown in Figure \ref{fig7}$(a)$. The plots of the RIGF of the logistic map  with respect to $\beta$ for $r=3.4$ (black line), $3.8$ (blue line), and $4.0$ (red line) are presented in Figure \ref{fig7}$(b)$ for  $\alpha=0.01$. From  \ref{fig7}$(b)$, we observe that the uncertainty for $r=4$ is greater than that for $r=3.4$ and $3.8$ when $\beta>1$ and they are equal for $\beta=1$, as expected. We also observe that the RIGF of  logistic map decreases when $\beta$ increases.

\begin{figure}[h!]
    	\centering
    	\subfigure[]{\label{c1}\includegraphics[height=2.9in]{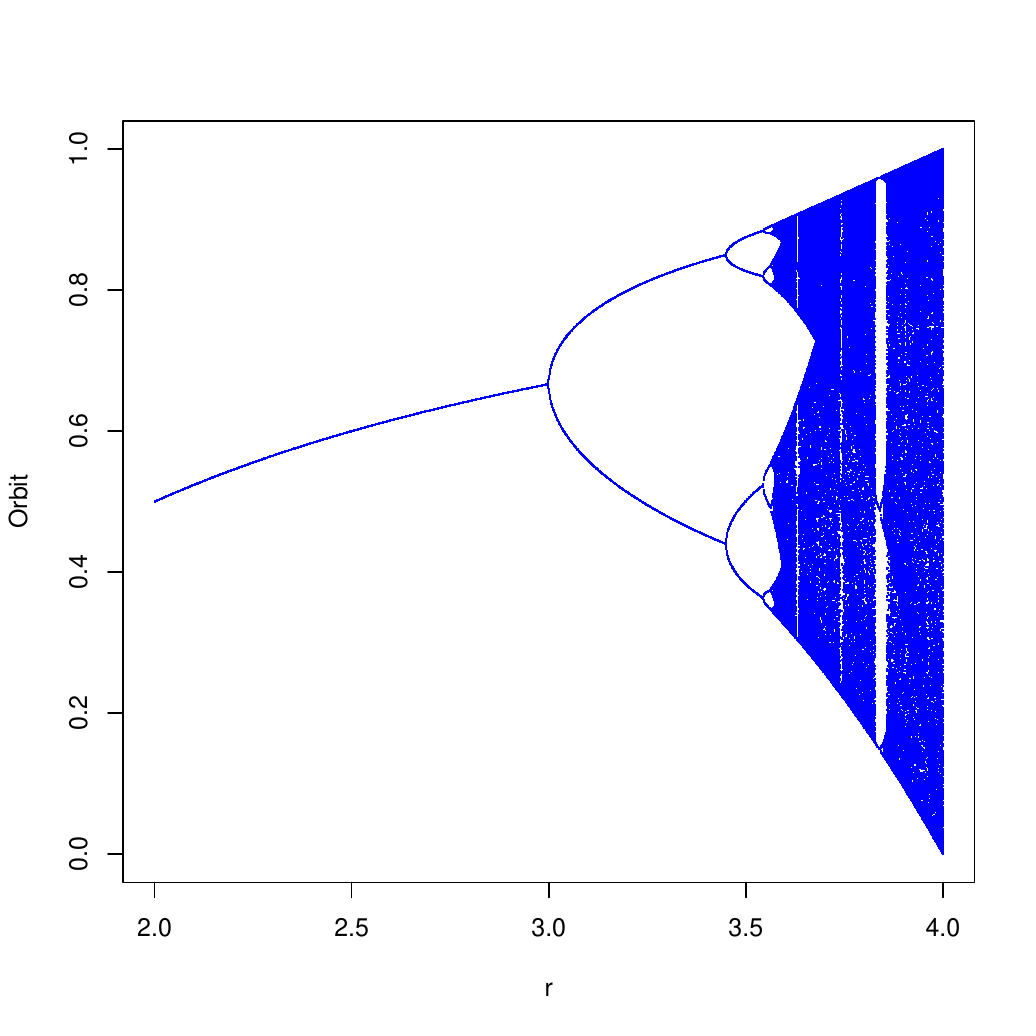}}
    	\subfigure[]{\label{c1}\includegraphics[height=2.9in]{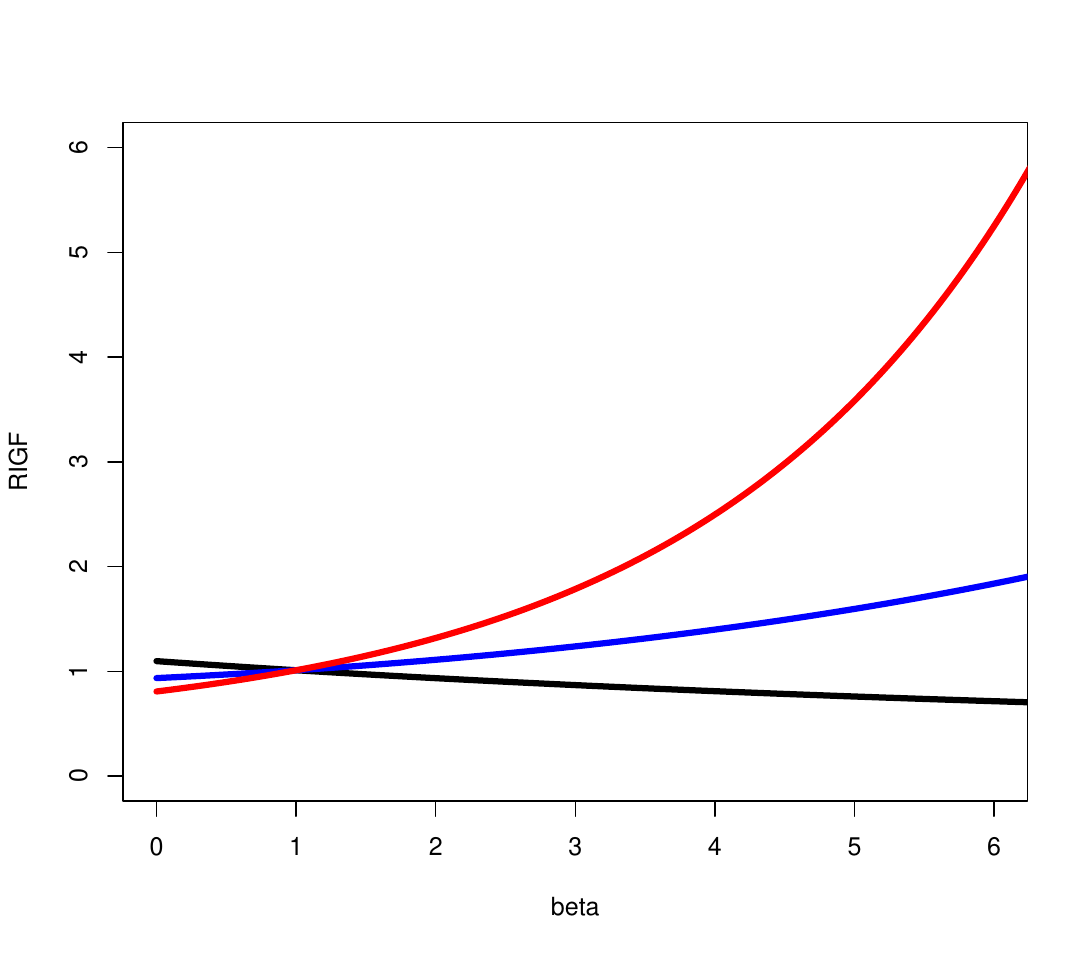}}
    		\caption{$(a)$ Bifurcation diagram of the logistic map in (\ref{eq6.12}) and $(b)$ the plots of the RIGF of logistic map for $r=4$ (red line), $r=3.8$ (blue line), and $r=3.4$ (black line).}
    	\label{fig7}	
    \end{figure}

\section{Concluding comments}\label{sec8}
In this paper, we have proposed some new information generating functions, which produce some well-known information measures, such as R\'enyi entropy, R\'enyi divergence, and Jensen-R\'enyi divergence measures. We have illustrated the generating functions with various examples. We have shown that the RIGF is shift-independent. Various bounds have been obtained as well. Further, the RIGF has been expressed in terms of the Shannon entropy of order $q>0$. We have obtained the RIGF for the escort distribution. It has been observed that the RDIGF reduces to the RIGF when the random variable $Y$ is uniformly distributed in the interval $(0,1)$. The RDIGF has been studied for generalized escort distribution. Further, the effect of this information generating function on monotone transformations has been discussed.
A kernel based non-parametric estimator and a parametric estimator of the RIGF have been proposed. A Monte Carlo simulation study has been conducted for both non-parametric and parametric estimators. The performance of the non-parametric as well as parametric estimators of the proposed RIGF and IGF has been studied based on the SD, AB, and MSE. Superior performance has been observed for the newly proposed estimator of RIGF. In addition, it has been shown that the parametric estimator performs better than the non-parametric estimator of the RIGF for the case of Weibull distribution, as one would expect.  Further, a real data set on the failure times (in minutes) of $15$ electronic components has been used for illustrative purpose. Few possible applications of the proposed RIGF and RDIGF have been explored. For three coherent systems, we have computed the values of the proposed RIGF, IGF, Renyi entropy, and varentropy. It has been observed that the RIGF has similar behaviour with other well-known measures. Further, a study regarding the usefulness of the RDIGF and RIGF as model selection criteria has been conducted. Finally, three chaotic maps have been considered and analysed to validate the use of the information generating functions introduced here.

\section*{Acknowledgements}  The authors would like to thank the Editor, Editorial Board Member: Prof. Min Xie, and  Referees for all their helpful comments and suggestions, which led to the substantial improvements. Shital Saha thanks the University Grants Commission (Award No. 191620139416), India, for financial assistantship to carry out this research work. The first two authors thank the research facilities provided by the Department of Mathematics, National Institute of Technology Rourkela, India.


\bibliography{refference}

\end{document}